\documentclass[reqno,11pt]{amsart}

\usepackage[top=2.0cm,bottom=2.0cm,left=3cm,right=3cm]{geometry}
\usepackage{amsthm,amsmath,amssymb,dsfont}
\usepackage{mathrsfs,amsfonts,functan,extarrows,mathtools}
\usepackage[colorlinks]{hyperref}

\usepackage{marginnote}
\usepackage{xcolor}
\usepackage{stmaryrd}
\usepackage{esint}
\usepackage{graphicx}
\usepackage{bm}

\usepackage{amsmath,latexsym,amsfonts,indentfirst,amsthm,amsxtra,amssymb,bm}

\usepackage{fancyhdr}

\usepackage{epstopdf}

\newtheorem{theorem}{Theorem}[section]

\newtheorem{remark}{Remark}[section]
\newtheorem{lemma}{Lemma}[section]

\numberwithin{equation}{section}
\allowdisplaybreaks

\def\tr{\mathrm{tr}}

\def\no{\nonumber}
\def\R{\mathbb{R}}
\def\T{\mathbb{T}}

\def\eps{\varepsilon}
\def\div{\mathrm{div}\,}

\def\J{\mathcal{J}}

\newcounter{wronumber}

\begin{document}
\title[Inertial Qian-Sheng model]
			{Entropy inequality and energy dissipation of inertial Qian-Sheng model for nematic liquid crystals}

\author[Ning Jiang]{Ning Jiang}
\address[Ning Jiang]{\newline School of Mathematics and Statistics, Wuhan University, Wuhan, 430072, P. R. China}
\email{njiang@whu.edu.cn}

\author[Yi-Long Luo]{Yi-Long Luo}
\address[Yi-Long Luo]
{\newline School of Mathematics, South China University of Technology, Guangzhou, 510641, P. R. China}
\email{luoylmath@scut.edu.cn}

\author[Yangjun Ma]{Yangjun Ma}
\address[Yangjun Ma]
        {\newline School of Mathematics and Statistics, Chongqing Jiaotong University,
        	Chongqing, 400074, P. R. China}
\email{yangjunma@whu.edu.cn}

\author[Shaojun Tang]{Shaojun Tang}
\address[Shaojun Tang]
        {\newline Department of Mathematics
        	Wuhan University of Technology,Wuhan, 430070, P. R. China}
\email{sjtang@ustc.edu.cn}

\thanks{\today}

\maketitle

\begin{abstract}
	For the inertial Qian-Sheng model of nematic liquid crystals in the $Q$-tensor framework, we illustrate the roles played by the entropy inequality and energy dissipation in the well-posedness of smooth solutions when we employ energy method. We first derive the coefficients requirements from the entropy inequality, and point out the entropy inequality is insufficient to guarantee energy dissipation. We then introduce a novel Condition (H) which ensures the energy dissipation. We prove that when both the entropy inequality and Condition (H) are obeyed, the local in time smooth solutions exist for large initial data. Otherwise, we can only obtain small data local solutions. Furthermore, to extend the solutions globally in time and obtain the decay of solutions, we require at least one of the two conditions: entropy inequality, or $\tilde{\mu}_2= \mu_2$, which significantly enlarge the range of the coefficients in previous works.\\
	
	
\end{abstract}

%\vspace*{10pt}

%\phantomsection
%\addcontentsline{toc}{section}{\contentsname}

%\tableofcontents

%%%%%%%%%%%%%%%%%%%%%%%%%%%%%%%%%%%%%%（正文）%%%%%%%%%%%%%%%%%%%%%%%%%%%%%
%%%%%%%%%%%%%%%%%%%%%%%%%%%%%%%%%%%%%%%%%%%%%%%%%%%%%%%%%%%%%%%%%%%%%%%%%%%

\section{Introduction}
\subsection{Inertial Qian-Sheng model}

We consider the hydrodynamics of nematic liquid crystal model in incompressible flow in the $Q$-tensor framework. For the introduction to basic $Q$-tensor theory of liquid crystals, see \cite{Mottram-2014-arXiv} and \cite{Zarnescu-TMMA-2012} for instance. In this paper we specifically study the incompressible inertial Qian-Sheng model proposed by Qian and Sheng in \cite{Qian-Sheng-PRE-1998}, which keeps the inertial effect and which is the Q-tensor version of the classic Ericksen-Leslie model established in \cite{Leslie-1968-ARMA} and particularly captures the biaxial alignment of the molecules, a feature not available in  Ericksen-Leslie model. Since it contains the inertial effect, which is represented by a second material derivate, it has a hyperbolic feature. So we also call it {\em hyperbolic Qian-Sheng model}. De Anna and Zarnescu \cite{DeAnna-Zarnescu-JDE-2018} provided the first global in time existence result for it. In \cite{DeAnna-Zarnescu-JDE-2018}, some relatively restrictive assumptions on the transport coefficients were made.

The main goal and novelty of this paper is to completely clarify the entropy inequality and introduce a new Condition (H) which ensure the energy dissipation. Furthermore, we classify the coefficients requirements for large initial data, small initial data local in time smooth solutions, and small initial data global in time solutions, respectively. As a consequence, we significantly improve the well-posedness  results in \cite{DeAnna-Zarnescu-JDE-2018}.

In order to clearly state our results, we first describe the system and its derivation. The local orientation of the molecules is described through a function $Q$ taking values from $ \Omega \subseteq \R^d$ $(d = 2 \ \textrm{or } 3)$ into the set of so-called $d$-dimensional $Q$-tensors, that is symmetric and traceless $d \times d$ matrices:
\begin{equation*}
	\begin{aligned}
		S_{0}^{(d)}: = \big\{ Q \in \mathbb{R}^{d \times d}; Q_{ij} = Q_{ji}, \tr(Q)=0 , i, j = 1, \cdots, d \big\} \,.
	\end{aligned}
\end{equation*}
The evolution of $Q$ is driven by the free energy of the molecules, as well as the transport distortion and alignment effects caused by the flow. The bulk velocity of the centers of masses of molecules obeys a forced incompressible Navier-Stokes system, with an additional stress tensor, a forcing term modeling the effect that the interaction of the molecules has on the dynamics of the center of masses. More precisely, the non-dimensional equation can be explicitly written as:
\begin{equation}\label{CIQS}
\begin{cases}
\partial_{t} u + u \cdot \nabla u - \tfrac{1}{2} \beta_{4} \Delta u +\nabla p = \div (\Sigma_{1} + \Sigma_{2} + \Sigma_{3} ), \\
\ \ \ \ \ \ \ \ \ \ \ \ \ \ \ \ \ \ \  \ \ \ \ \div u =0\,, \\
J\ddot{Q} + \mu_{1} \dot{Q}= L \Delta Q - a Q + b(Q^{2}- \tfrac{1}{d} | Q |^{2}I_{d})-c Q | Q |^{2} + \tfrac{\tilde{\mu}_{2}}{2} A + \mu_{1}[\Omega, Q] \,,
\end{cases}
\end{equation}
with the forms of the notations $\Sigma_{i}\ (i=1,2,3)$
\begin{equation}\label{Sigma1-4}
\begin{aligned}
\Sigma_{1} & = - L \nabla Q \odot\nabla Q  \,, \\
\Sigma_{2}& = \beta_{1} Q  \tr(Q A)+ \beta_{5} A Q + \beta_{6} Q A \,, \\
\Sigma_{3} &= \tfrac{\mu_{2}}{2} (\dot{Q} - [\Omega, Q]) + \mu_{1} [Q, (\dot{Q} - [\Omega, Q])] \,.
\end{aligned}
\end{equation}
The symbol $ \dot{f}= ( \partial_{t} + u \cdot\nabla ) f $ denotes the {\em material derivative} and for any two $ d \times d$ matrices $ M , N $,
we denote their commutator as $ [ M , N ] := MN - NM $. We also denote the inner product on the space  of matrices as $ M : N =\tr (MN)$,
$ | M |$ denotes the Frobenius  norm of the matrix, i.e. $ | M |= \sqrt{ M : M^{\top} }$.
Furthermore, we denote $A_{ij} :=\tfrac{1}{2} (\partial_{j} u_{i} + \partial_{i} u_{j}) $,
$ \Omega_{ij} : = \tfrac{1}{2} (\partial_{j} u_{i} - \partial_{i} u_{j} ) $ for $ i, j= 1,\ldots ,d $ and the Ericksen tensor $\nabla Q \odot \nabla Q$ are $ (\nabla Q \odot \nabla Q)_{ij} = \Sigma_{k,l=1}^{d}
\partial_{i} Q_{kl} \partial_{j} Q_{kl} $, where we employ the notation $\partial_i f = \tfrac{\partial f}{\partial x_i}$ for a scalar function $f$.
The meanings of the constants, the physical relevance of the equations and there meanings will be illustrated in the next subsection. Furthermore, the initial data of \eqref{CIQS} is imposed on
\begin{equation}\label{initial date}
\begin{aligned}
(u, Q, \dot{Q}) (x,t)|_{t=0} = (u^{in}, Q^{in}, \tilde{Q}^{in})(x) \in \mathbb{R}^d \times S_{0}^{(d)} \times S_{0}^{(d)} \,.
\end{aligned}
\end{equation}

%%%%%%%%%%%%%%%%%%%%%%%%%%%%%%%%%%%%%%%%%%%%%%%%%%%%%%%%%%%%%%%%%%%%%%%%%%%%%%%%%%%%%%%%%%%%%%

\subsection{Derivation of inertial Qian-Sheng model}\label{Sec:Phys-Derive}

In this subsection, we state the formal derivation of the Qian-Sheng model \eqref{CIQS}, which was proposed in \cite{Qian-Sheng-PRE-1998}. For the convenience of readers and more importantly, to illustrate the entropy inequality, we write down in some details for these the derivations. In the following we take the domain $\Omega$ to be $\R^d$. For a material volume $V_{ol} \subset \R^d$ bounded by a surface $A_{rea}$, we consider the following conservation laws of energy and angular moment (here and in the following we use the Einstein summation convention, of summation over repeated indexes):
\begin{equation}\label{con-energy}
\begin{aligned}
\tfrac{D}{D t} \int_{V_{ol}} (\tfrac{1}{2} \rho u_i u_i + \tfrac{1}{2} J \dot{Q}_{ij} \dot{Q}_{ij} + U) d V_{ol} \\
= \int_{V_{ol}} \rho r d V_{ol} + \int_{A_{rea}} ( t_i u_i + s_{ij} \dot{Q}_{ij} - h ) d A_{rea} \,,
\end{aligned}
\end{equation}
and
\begin{equation}\label{con-angular}
\begin{aligned}
\tfrac{D}{D t} \int_{V_{ol}} J \dot{Q}_{\alpha \beta} d V_{ol} = \int_{V_{ol}} g_{\alpha \beta} d V_{ol} + \int_{A_{rea}} s_{\alpha \beta} d A_{rea} \,,
\end{aligned}
\end{equation}
with an entropy production inequality of the form
\begin{equation}\label{Entropy1}
\begin{aligned}
\tfrac{D}{D t} \int_{V_{ol}} \rho S d V_{ol} - \int_{V_{ol}} \frac{\rho r}{T} d V_{ol} + \int_{A_{rea}} K d A_{rea} \geq 0 \,,
\end{aligned}
\end{equation}
where $\frac{D}{D t}$ also denotes the material time derivative with respect to the velocity $u$, $\rho$ is the constant density, $U$ is the internal energy per unit mass, $r$ is the heat supply per unit mass, per unit time, $t_i$ is the surface force per unit area, $s_{\alpha \beta}$ is the director surface force associated with the $Q$-tensor framework per unit area, $h$ is the flux of heat out of the volume per unit area, per unit time, and $g_{\alpha \beta}$ is an intrinsic director body force per unit volume in the $Q$-tensor regime. $S$ is the entropy per unit mass, $T$ is the temperature and $K$ is the flux of entropy out of the volume per unit area, per unit time. Moreover,
\begin{equation}\label{J}
\begin{aligned}
J > 0
\end{aligned}
\end{equation}
is the inertial constant, which captures the inertial effect of the liquid molecules.

Following almost the same arguments in Section 3 and Section 4 of \cite{Leslie-1968-ARMA}, we roughly sketch the outline of derivation of the inertial Qian-Sheng model \eqref{CIQS}. From the equations \eqref{con-energy} and \eqref{con-angular}, we have
\begin{equation}\label{Derivation-1}
\begin{aligned}
& \partial_i u_i = 0 \,, \quad t_i = \sigma_{ij} \nu_j \,, \quad \rho \dot{u}_i = \partial_j ( - p \delta_{ij} + \sigma_{ij} ) \,, \\
& s_{\alpha \beta} = \pi_{\gamma \alpha \beta } \nu_\gamma \,, \quad J \ddot{Q}_{\alpha \beta} = \partial_\gamma \pi_{\gamma \alpha \beta} + g_{\alpha \beta} \,,
\end{aligned}
\end{equation}
where $\sigma_{ij}$ are the components of the surface force across the $x_j$-planes, $\nu_j$ is the unit normal to the surface, $\pi_{ \gamma \alpha \beta }$ are the components of the director surface force across the $x_\gamma$-planes, $p$ is the pressure of the incompressible flow (i.e., the Lagrangian multiplier of the incompressibility), and $\delta_{ij}$ is the standard Kronecker delta symbol. Furthermore, \eqref{con-energy} also implies that
\begin{equation}\label{Derivation-2}
\begin{aligned}
h = q_i \nu_i \,, \quad \rho \dot{U} = \rho r - \partial_i q_i + \sigma_{ij} A_{ij} + \pi_{ \gamma \alpha \beta } N_{\alpha \beta \gamma} - g_{\alpha \beta} \mathscr{N}_{\alpha \beta} \,,
\end{aligned}
\end{equation}
where $q_i$ is the heat flux across the $x_i$-planes, where $\mathscr{N}$ stands for the co-rotational time flux of $Q$, whose $(\alpha, \beta)$-th component is defined as
\begin{equation}\label{General-4}
\begin{aligned}
\mathscr{N}_{\alpha \beta} := ( \dot{Q} - [\Omega , Q] )_{\alpha \beta} = \dot{Q}_{\alpha \beta} - \Omega_{\alpha l} Q_{l \beta} + Q_{\alpha l} \Omega_{l \beta} \,,
\end{aligned}
\end{equation}
and $N_{\alpha \beta \gamma} = \partial_\gamma \dot{Q}_{\alpha \beta} - [ \Omega, \partial_\gamma Q ]_{\alpha \beta}$. Here $\tr \mathscr{N} = \mathscr{N}_{\alpha \alpha} = 0$ (since $\tr \dot{Q} = 0$) and $\mathscr{N}$ represents the time rate of change of $Q_{\alpha \beta}$ with respect to the background fluid angular velocity $\omega = \tfrac{1}{2} \nabla \times u$.

The entropy inequality \eqref{Entropy1} implies that
\begin{equation}\label{Derivation-3}
\begin{aligned}
K = p_i \nu_i \,, \quad \rho T \dot{S} - \rho r + T \partial_i p_i \geq 0 \,,
\end{aligned}
\end{equation}
where $p_i$ is the flux of entropy across the $x_i$-planes. If we introduce a Helmholtz free energy $\mathscr{F}$ and a vector $\varphi_i$, where
\begin{equation}\label{Derivation-4}
\begin{aligned}
\mathscr{F} = U - T S \,, \quad \varphi_i = q_i - T p_i \,.
\end{aligned}
\end{equation}
We thereby derive from \eqref{Derivation-2}, \eqref{Derivation-3} and \eqref{Derivation-4} that
\begin{equation}\label{Derivation-5}
\begin{aligned}
\sigma_{ij} A_{ij} + \pi_{ \gamma \alpha \beta } N_{\alpha \beta \gamma} - g_{\alpha \beta} \mathscr{N}_{\alpha \beta} - p_i \partial_i T - \rho ( \dot{\mathscr{F}} + S \dot{T} ) - \partial_i \varphi_i \geq 0 \,.
\end{aligned}
\end{equation}
Noticing that $S = - \frac{\partial \mathscr{F}}{\partial T}$, from the similar arguments of constitutive equations in Section 4 of \cite{Leslie-1968-ARMA}, the entropy inequality \eqref{Derivation-5} will be
\begin{equation}\label{Derivation-6}
\begin{aligned}
\Big( \sigma_{ij} + \rho^2 \frac{\partial \mathscr{F}}{\partial \rho} \delta_{ij} + \rho \frac{\partial \mathscr{F}}{\partial ( \partial_i Q_{\alpha \beta} )} \partial_j Q_{\alpha \beta} \Big) A_{ij} - \Big( g_{\alpha \beta} + \rho \frac{\partial \mathscr{F}}{\partial Q_{\alpha \beta}} \Big) \mathscr{N}_{\alpha \beta} \\
- \big( p_i + \frac{\partial \varphi_i}{\partial T} \big) \partial_i T - \frac{\partial \varphi_i}{\partial Q_{\alpha \beta}} \partial_i Q_{\alpha \beta} \geq 0
\end{aligned}
\end{equation}
and
\begin{equation}\label{Derivation-7}
\begin{aligned}
\pi_{ \gamma \alpha \beta } = \rho \frac{\partial \mathscr{F}}{\partial ( \partial_\gamma Q_{\alpha \beta} )} + \frac{\partial \varphi_\gamma}{\partial \mathscr{N}_{\alpha \beta}} \,.
\end{aligned}
\end{equation}

Since the flow that we consider in this paper is isothermal and incompressible, i.e., $T$ is taken as a fixed constant and $\partial_i u_i = 0$, it holds that
\begin{equation*}
	\begin{aligned}
		\varphi_i = 0 \,, \quad p_i = \frac{q_i}{T} \,, \quad \partial_i T = 0  \,, A_{ii} = \partial_i u_i = 0\,.
	\end{aligned}
\end{equation*}
Then \eqref{Derivation-6} and \eqref{Derivation-7} can be rewritten as
\begin{equation}\label{Derivation-8}
\begin{aligned}
\Big( \sigma_{ij} + \rho \frac{\partial \mathscr{F}}{\partial ( \partial_i Q_{\alpha \beta} )} \partial_j Q_{\alpha \beta} \Big) A_{ij} - \Big( g_{\alpha \beta} + \rho \frac{\partial \mathscr{F}}{\partial Q_{\alpha \beta}} \Big) \mathscr{N}_{\alpha \beta} \geq 0
\end{aligned}
\end{equation}
and
\begin{equation}\label{Derivation-9}
\begin{aligned}
\pi_{ \gamma \alpha \beta } = \rho \frac{\partial \mathscr{F}}{\partial ( \partial_\gamma Q_{\alpha \beta} )} \,.
\end{aligned}
\end{equation}
We denote the values of the stress and intrinsic body force in static isothermal deformation by
\begin{equation}\label{Derivation-10}
\begin{aligned}
\sigma_{ij}^0 = - \rho \frac{\partial \mathscr{F}}{\partial ( \partial_i Q_{\alpha \beta} )} \partial_j Q_{\alpha \beta} \,, \quad g_{\alpha \beta}^0 = - \rho \frac{\partial \mathscr{F}}{\partial Q_{\alpha \beta}} \,,
\end{aligned}
\end{equation}
respectively, and employ the notations
\begin{equation}\label{Derivation-11}
\begin{aligned}
\sigma_{ij} = \sigma_{ij}^0 + \sigma_{ij}' \,, \quad g_{\alpha \beta} = g_{\alpha \beta}^0 - \lambda \delta_{\alpha \beta} + g_{\alpha \beta}'  \,,
\end{aligned}
\end{equation}
where $\lambda \in \R$ is the Lagrange multiplier enforcing the tracelessness and symmetry of the $Q$-tensor. One notices that $\lambda \delta_{\alpha \beta} \mathscr{N}_{\alpha \beta} = 0$ for any $\lambda \in \R$, since $\tr \mathscr{N} = 0$. Then the entropy inequality \eqref{Derivation-8} will be
\begin{equation}\label{Entropy2}
\begin{aligned}
\sigma'_{ij} A_{ij} - g'_{\alpha \beta} \mathscr{N}_{\alpha \beta} \geq 0 \,.
\end{aligned}
\end{equation}
In summary, the equations \eqref{Derivation-1}, \eqref{Derivation-9}, \eqref{Derivation-10} and \eqref{Derivation-11} imply that $(\rho, u, Q)$ subjects to the system
\begin{equation}\label{General-1}
\left\{
\begin{array}{c}
\rho \dot{u}_i = \partial_j ( - p \delta_{ij} + \sigma_{ij} ) \,, \quad \partial_i u_i = 0\,, \\[2mm]
J \ddot{Q}_{\alpha \beta} = \partial_\gamma \pi_{ \gamma \alpha \beta } + g_{\alpha \beta} \,,
\end{array}
\right.
\end{equation}
where
\begin{equation}\label{General-2}
\begin{aligned}
& \sigma_{ij} = - \rho \frac{\partial \mathscr{F}}{\partial ( \partial_i Q_{\alpha \beta} )} \partial_j Q_{\alpha  \beta} + \sigma'_{ij} \,, \\
& \pi_{\gamma \alpha \beta} = \rho \frac{\partial \mathscr{F}}{\partial ( \partial_\gamma Q_{\alpha \beta} )} \,, \\
& g_{\alpha \beta} = - \rho \frac{\partial \mathscr{F}}{\partial Q_{\alpha \beta}} + g'_{\alpha \beta} - \lambda \delta_{\alpha \beta} \,.
\end{aligned}
\end{equation}

Since the general incompressible inertial $Q$-tensor model for liquid crystals \eqref{General-1}-\eqref{General-2}-\eqref{Entropy2}-\eqref{General-4} is too complicated, we consider a simplified model by specializing some quantities. More precisely, we first take the simplest form of $\mathscr{F}$ such that $\rho \mathscr{F}$ is the Landau-de Gennes free energy density, namely,
\begin{equation}
\begin{aligned}
\rho \mathscr{F} = \tfrac{L}{2} |\nabla Q|^2 + \psi_B (Q) \,,
\end{aligned}
\end{equation}
which models the spatial variations through the $\tfrac{L}{2} |\nabla Q|^2$ term with positive diffusion coefficient
\begin{equation}\label{Diff-Coeff-L}
\begin{aligned}
L > 0 \,,
\end{aligned}
\end{equation}
and the nematic ordering enforced through the ``bulk term" $\psi_B (Q)$ taken to be of the following standard form (see \cite{Mottram-2014-arXiv}, for instance)
\begin{equation}
\begin{aligned}
\psi_B (Q) = \tfrac{a}{2} \tr (Q^2) - \tfrac{b}{3} \tr (Q^3) + \tfrac{c}{4} \big( \tr (Q^2) \big)^2
\end{aligned}
\end{equation}
with the phenomenological material constants
\begin{equation}\label{Coeffs-abc}
\begin{aligned}
a > 0 \,, b, c \in \R \,.
\end{aligned}
\end{equation}
Second, the viscous stress $\sigma'$ can be taken as
\begin{equation}\label{sigma-prime}
\begin{aligned}
\sigma'_{ij} : = \beta_1 Q_{ij} Q_{lk} A_{lk} + \beta_4 A_{ij} & + \beta_5 Q_{jl} A_{li} + \beta_6 Q_{il} A_{lj} \\
& + \mu_1 Q_{il} \mathscr{N}_{lj} - \mu_1 \mathscr{N}_{il} Q_{lj} + \tfrac{1}{2} \mu_2 \mathscr{N}_{ij} \,,
\end{aligned}
\end{equation}
and the viscous molecular field $g'$ can be given by
\begin{equation}
\begin{aligned}
g'_{\alpha \beta} : = \tfrac{1}{2} \tilde{\mu}_2 A_{\alpha \beta} - \mu_1 \mathscr{N}_{\alpha  \beta} \,,
\end{aligned}
\end{equation}
where $\beta_1$, $\beta_4$, $\beta_5$, $\beta_6$, $\mu_1$, $\mu_2$ and $\tilde{\mu}_2$ are viscosity coefficients subjecting to well-known {\em Parodi's relation}
\begin{equation}\label{Parodi-Rlt}
\begin{aligned}
\beta_6 - \beta_5 = \mu_2 \,.
\end{aligned}
\end{equation}
We emphasize that, in the above setting, the Lagrange multiplier $\lambda$ in \eqref{General-2} can be easily determined as $\lambda = - \tfrac{1}{d} b |Q|^2 I_d $ (with $I_d$ the $d \times d$ identity matrix). Let the constant density $\rho = 1$. Consequently, we deduce the incompressible inertial Qian-Sheng model \eqref{CIQS}.

\subsection{Entropy inequality and energy dissipation}
\subsubsection{Entropy inequality}

Now we clarify the coefficients implied by entropy inequality \eqref{Entropy2}. From the symmetry of $\mathscr{N}$, $Q$ and the facts $[Q, \mathscr{N}] : A = 0$ implied by the skew-symmetric property of $[Q, \mathscr{N}]$, we deduce that
\begin{equation*}
	\begin{aligned}
		\sigma'_{ij} A_{ij} - g'_{\alpha \beta} \mathscr{N}_{\alpha \beta} = & \beta_1 (Q : A)^2 + \beta_4 |A|^2 + (\beta_5 + \beta_6) \tr (QA^2) \\
		& + \mu_1 |\mathscr{N}|^2 + \tfrac{1}{2} ( \tilde{\mu}_2 - \mu_2 ) A : \mathscr{N} \,,
	\end{aligned}
\end{equation*}
which infer the total entropy
\begin{equation}\label{Entropy3}
\begin{aligned}
& \int_{\R^d} \! ( \sigma'_{ij} A_{ij} - g'_{\alpha \beta} \mathscr{N}_{\alpha \beta} ) d x \\
= & \beta_1 \int_{\R^d} \! ( Q: A )^2 d x + \tfrac{1}{2} \beta_4 \int_{\R^d} \! |\nabla u|^2 d x + (\beta_5 + \beta_6) \int_{\R^d} \tr (Q A^2) d x \\
& + \mu_1 \int_{\R^d} |\dot{Q} - [\Omega , Q]|^2 d x + \tfrac{1}{2} ( \tilde{\mu}_2 - \mu_2 ) \int_{\R^d} A : ( \dot{Q} - [\Omega , Q] ) d x \,.
\end{aligned}
\end{equation}
Here the relation
\begin{equation*}
	\begin{aligned}
		\int_{\R^d} \beta_4 |A|^2 d x = \tfrac{1}{2} \beta_4 \int_{\R^d} \! |\nabla u|^2 d x
	\end{aligned}
\end{equation*}
has been utilized, which derives from the incompressibility $\div u = 0$ and the integration by parts. The entropy inequality \eqref{Entropy2} requires the right-hand side of \eqref{Entropy3} must be non-negative for all functions $u$ and $Q$. We claim that this requirement is fulfilled {\em if and only if} the coefficients satisfy
\begin{align}
	\label{Entropy-Coeffs} & \beta_1 \geq 0 \,, \ \beta_4 > 0 \,, \ \mu_1 > 0 \,, \\
	\label{Entropy-Coeffs-1} &  \qquad \qquad \beta_5 + \beta_6 = 0 \,,
\end{align}
and
\begin{equation}\label{H}
(\tilde{\mu}_2-\mu_2)^2 < 8\beta_4\mu_1\,,
\end{equation}
which is equivalent to the fact: there are generic constants $\delta_0 , \delta_1 \in (0,1]$ such that
\begin{equation}\label{Entropy-Coeffs-2}
\begin{aligned}
\tfrac{1}{2} \beta_4 |X|^2 + \mu_1 |Y|^2 +  \tfrac{1}{2} ( \tilde{\mu}_2 - \mu_2 ) X:Y \geq \delta_0 \tfrac{1}{2} \beta_4 |X|^2 + \delta_1 \mu_1 |Y|^2
\end{aligned}
\end{equation}
for all matrices $X, Y \in \R^{d \times d}$. We explain these coefficients requirements in detail as following:
\begin{itemize}
	\item The requirement $\beta_1 \geq 0 $ is apparent;
	\item The integrand $\tr (Q A^2)$ is un-signed, so the only way to ensure this term non-negative is to settle $\beta_5 + \beta_6 = 0\,;$
	\item The integrand $A : ( \dot{Q} - [\Omega , Q] ) = \nabla u : ( \dot{Q} - [\Omega , Q] )$ is un-signed, but it can be controlled by $|A|^2$ and $|\dot{Q} - [\Omega , Q]|^2$, which appear in the 2nd and 4th terms of the right-hand side of \eqref{Entropy3}. So we require $\beta_4 > 0 \,, \ \mu_1 > 0 $ and the condition \eqref{Entropy-Coeffs-2}. We will show below that \eqref{Entropy-Coeffs-2} is equivalent to condition \eqref{H}.
\end{itemize}
Note $u$ and $Q$ are arbitrary, so the above conditions are also necessary.

In fact, the condition \eqref{Entropy-Coeffs-2} reads
\begin{equation*}
	\begin{aligned}
		(1 - \delta_0) \tfrac{1}{2}  \beta_4 |X|^2 + (1 - \delta_1) \mu_1 |Y|^2 + \tfrac{1}{2} (\tilde{\mu}_2 - \mu_2) X:Y \geq 0
	\end{aligned}
\end{equation*}
for all $X, Y \in \R^{d \times d}$, which is equivalent to $ \tfrac{1}{4} ( \tilde{\mu}_2 - \mu_2 )^2 - 4 (1 - \delta_0) (1 - \delta_1) \tfrac{1}{2} \beta_4 \mu_1 \leq 0 $. In other words,
\begin{equation}\label{Lower-Bnd}
\begin{aligned}
(1 - \delta_0) (1 - \delta_1)  \geq \tfrac{(\tilde{\mu}_2 - \mu_2)^2}{8\beta_4 \mu_1 } > 0 \,,
\end{aligned}
\end{equation}
which is apparently equivalent to condition \eqref{H}. We remark that condition \eqref{H} or \eqref{Entropy-Coeffs-2} does not appear in the previous literatures for $Q$ tensor system of liquid crystals.

In the original paper \cite{Qian-Sheng-PRE-1998} of Qian and Sheng, they set $\tilde{\mu}_2 =\mu_2$, which can be derived from the Onsager theorem \cite{deGroot-Mazur-1962-BOOK}. In this case we can take $\delta_0 = \delta_1 = 1$ in \eqref{Entropy-Coeffs-2} and the condition \eqref{H} is automatically satisfied. Some special case of $\tilde{\mu}_2\neq \mu_2$ was also considered, for example, $\tilde{\mu}_2 = -\mu_2 $ in \cite{PopaNita-Oswald-2002-PRE, PopaNita-Sluckin-Kralj-2005-PRE}.

\subsubsection{Energy dissipation}

However, the entropy inequality does not guarantees the energy dissipation which is crucial in well-posedness. In the {\em a priori} estimate, as shown in Lemma \ref{priori}, the main term that should provide dissipation is
\begin{equation*}
	\begin{aligned}
		\tfrac{1}{2} \beta_4 \| \nabla u \|^2_{L^2} + \mu_1 \| \dot{Q} \|^2_{L^2} + \mu_1 \| [\Omega , Q] \|^2_{L^2} - \mu_2 \langle A, [\Omega, Q] \rangle - \tfrac{1}{2} ( \tilde{\mu}_2 - \mu_2 ) \langle A , \dot{Q} \rangle \,.
	\end{aligned}
\end{equation*}
Under the entropy inequality \eqref{Entropy2}, or its equivalent requirements on the coefficients \eqref{Entropy-Coeffs}, \eqref{Entropy-Coeffs-1} and \eqref{H}, the above quantity is not guaranteed to be positive. Some more restrictions on these coefficients are necessary.  Thus, we give the following condition, which is an improved version of \eqref{Entropy-Coeffs-2}. In fact, we require that there are two constants $\delta_0, \delta_1 \in (0, 1]$ such that
\begin{equation}\label{Condition-H}
\begin{aligned}
F(X, Y, Z) : = & \tfrac{1}{2} \beta_4 |X|^2 + \mu_1 ( |Y|^2 + |Z|^2 ) - \tfrac{1}{2} ( \tilde{\mu}_2 - \mu_2 ) X:Y - \mu_2 X:Z \\
\geq & \delta_0 \tfrac{1}{2} \beta_4 |X|^2 + \delta_1 \mu_1 |Y|^2
\end{aligned}
\end{equation}
for all $X, Y, Z \in \R^{d \times d}$. We rewrite \eqref{Condition-H} as
\begin{equation}
(1-\delta_0)\tfrac{1}{2}\beta_4|X|^2 + (1-\delta_1)\mu_1 |Y|^2 + \mu_1|Z|^2 - \tfrac{1}{2}(\tilde{\mu}_2 -\mu_2)X:Y-\mu_2 X:Z \geq 0\,,
\end{equation}
for all vector $X$, $Y$ and $Z$. This is equivalent to the Hessian matrix of this quadratic form is non-negative, i.e.
\begin{equation}
\begin{aligned}
\left(
\begin{array}{ccc}
(1-\delta_0)\beta_4 & -\tfrac{1}{2}(\tilde{\mu}_2- \mu_2) & -\mu_2\\
-\tfrac{1}{2}(\tilde{\mu}_2- \mu_2) & 2\mu_1(1-\delta_1) & 0\\
-\mu_2 & 0 & 2\mu_1\\
\end{array}
\right) \geq 0\,,
\end{aligned}
\end{equation}
which is further equivalent to the following three conditions:
\begin{equation}\label{CH-1}
\beta_4> 0\,,\quad\!\mbox{and}\quad\! (\tilde{\mu}_2-\mu_2)^2 < 8\beta_4\mu_1\,,
\end{equation}
and
\begin{equation}\label{CH-2}
(\tilde{\mu}_2-\mu_2)^2 + 4 \mu_2^2< 8\beta_4\mu_1\,.
\end{equation}
The first inequality of \eqref{CH-1} is already included in \eqref{Entropy-Coeffs}-\eqref{Entropy-Coeffs-1}, and the second inequality is exactly \eqref{H}, which is weaker than \eqref{CH-2}. We call \eqref{CH-2} as {\bf Condition (H)}, which is equivalent to \eqref{Condition-H} in practice.

We collect the other coefficients' relations \eqref{J}, \eqref{Diff-Coeff-L}, \eqref{Coeffs-abc}, \eqref{Parodi-Rlt} and \eqref{Entropy-Coeffs} here, namely,
\begin{equation}\label{Coeffs-Rlt-all}
\begin{aligned}
& a > 0 \,, \ b, c \in \R \,, \ J > 0 \,, \ L > 0 \,, \\
& \beta_1 \geq 0 \,, \ \beta_6 - \beta_5 = \mu_2 \,,\\
& \beta_4 > 0 \,, \ \mu_1>0 \,.
\end{aligned}
\end{equation}

%%%%%%%%%%%%%%%%%%%%%%%%%%%%%%%%%%%%%%%%%%%%%%%%%%%%%%%%%%%%%%%%%%%%%%%%%%%%%%%%%%%%%%%%%%%%%%

%%%%%%%%%%%%%%%%%%%%%%%%%%%%%%%%%%%%%%%%%%%%%%%%%%%%%%%%%%%%%%%%%%%%%%%%%%%%%%%%%%%%%%%%%%%%%%

\subsection{Notations and main results}
For the sake of convenience, we first introduce some notations throughout this paper.
We denote by $A \lesssim B $ if there exists a constant $C > 0 $,
such that $ A \leq  C B$.  For convenience, we also denote $L^{p} := L^{p}( \mathbb{R}^{d})$ by the standard $L^p$ space for all $ p \in [1 , +\infty ] $. For $p=2$, we use the notation $\langle \cdot , \cdot \rangle $ to represent
the inner product on the Hilbert space $ L^{2} $.

For any multi-index $ k=( k_1, k_2, \cdots, k_d) $ in $ \mathbb{N}^d$, we denote the $ k^{th} $ partial derivative operator by
\begin{equation}\no
\begin{aligned}
\partial^{k} = \partial^{k_1}_{x_1}  \partial^{k_2}_{x_2} \cdots \partial^{k_d}_{x_d} \,.
\end{aligned}
\end{equation}
We employ the notation $k \leq k'$ to represent that every component of $k \in \mathbb{N}^d$ is not greater than that of $k' \in \mathbb{N}^d$. Moreover, $k < k'$ means that $k \leq k'$ and $|k| < |k'|$, where $|k| = k_1 + k_2 + \cdots + k_d \in \mathbb{N}$. We now define the following two Sobolev spaces $H^s$ and $\dot{H}^s$ endowed with the norms
\begin{equation}\nonumber
\begin{aligned}
\| f \|_{H^{s}} = \big( \sum^{s}_{|k|=0} \| \partial^{k} f  \|^{2}_{L^2} \big)^{\frac{1}{2}} \,, \quad \| f \|_{\dot{H}^{s}} = \big( \sum^{s}_{|k|=1} \| \partial^{k} f \|^{2}_{L^2} \big)^{\frac{1}{2}} \,.
\end{aligned}
\end{equation}
respectively.

We now introduce the following energy functional $E(t)$ as
\begin{equation}\label{Ener}
\begin{aligned}
E(t)= \| u \|^{2}_{H^{s}} + J \|\dot{Q} \|^{2}_{H^{s}}
+L \| \nabla Q \|^{2}_{H^{s}} +a \| Q \|^{2}_{H^{s}} \,.
\end{aligned}
\end{equation}
In particular, the initial energy is given as
\begin{equation}\label{IC-Ener-1}
\begin{aligned}
E^{in} =  \| u^{in} \|^{2}_{H^{s}}
+ J \| \tilde{Q}^{in}\| ^{2}_{H^{s}}
+L \| \nabla Q^{in} \|^{2}_{H^{s}} + a \| Q^{in} \|^{2}_{H^{s}} \,.
\end{aligned}
\end{equation}

Then, the main results of this paper are stated as follows:

\begin{theorem}[Large data local well-posedness]\label{Main-Thm-1}
	Let the integer $ s > \frac{d}{2} + 1 $ $(d=2 \textrm{ or } 3)$  and the coefficients satisfy \eqref{Coeffs-Rlt-all}. If $\beta_5+\beta_6=0$, and the Condition (H) \eqref{CH-2} are satisfied, and the initial energy $E^{in} < \infty$,  Then there exists a $ T > 0$, depending only on the initial data, the all coefficients, $s$ and $d$, such that the Cauchy problem \eqref{CIQS}-\eqref{initial date} admits a unique solution $(u, Q)$ satisfying
	\begin{equation}\label{Regularities-Solt}
	\begin{aligned}
	u \in& L^{\infty} (0, T; H^{s} ( \mathbb{R}^d ) ) \cap L^{2} (0, T; \dot{H}^{s+1} ( \mathbb{R}^d ) ) \,, \\
	\dot{Q} \in & L^{\infty} (0, T; H^{s} ( \mathbb{R}^d ) ) \cap L^{2} (0, T; H^{s} ( \mathbb{R}^d ) ) \,,\\
	Q \in&  L^{\infty} (0, T; H^{s+1} ( \mathbb{R}^d ) ) \,,
	\end{aligned}
	\end{equation}
	and subjecting to the following energy bound
	\begin{equation}\label{Enrg-Bnd-Loc}
	\begin{aligned}
	& \sup_{ 0 \leq t \leq T} \Big(  \| u\|^{2}_{H^{s}} + \| \dot{Q} \|^{2}_{H^{s}} + \| \nabla Q \|^{2}_{H^{s}} + \| Q \|^{2}_{H^{s}} \Big) \\
	& + \int^{T}_{0} \Big\{ \| \nabla u  \|^{2}_{H^{s}} + \sum_{|k| \leq s} \big(  \| Q : \partial^{k} A \|^{2}_{L^2} + \| \partial^{k} \dot{Q} \|^{2}_{L^{2}}  + \| [ \partial^{k} \Omega , Q ] \|^{2}_{L^{2}} \big) \Big\} d t \leq C_{0} \,,
	\end{aligned}
	\end{equation}
	where the constant  $ C_0 > 0$ depends only on the initial data, the all coefficients, $s$, $d$ and $ T $.
\end{theorem}
When the relation $\beta_5+ \beta_6=0$ which comes from entropy inequality and Condition (H) \eqref{CH-2} are NOT SIMULTANEOUSLY satisfied, even for local in time solution, the initial data is required to be small. This is the following theorem.

\begin{theorem}[Small data local well-posedness]\label{Main-Thm-1-prime}
	Let the integer $ s > \frac{d}{2} + 1 $ $(d=2 \textrm{ or } 3)$  and the coefficients satisfy \eqref{Coeffs-Rlt-all}. If one of the relation $\beta_5+ \beta_6 \neq 0$ or $(\tilde{\mu}_2-\mu_2)^2 + 4 \mu_2^2 \geq  8\beta_4\mu_1$ is satisfied, then there exists a small number $\eps_1>0$, depending on the coefficients, $s$ and $d$, such that, if the initial energy $E^{in} < \eps_1$, then the local in time existence and energy bound in Theorem \ref{Main-Thm-1} are also hold.
\end{theorem}

Then we study the global existence of classical solutions to the system \eqref{CIQS}.

\begin{theorem}[Global well-posedness]\label{Main-Thm-2}
	Let the integer $ s > \frac{d}{2} + 1 $ $(d=2 \textrm{ or } 3)$  and the coefficients satisfy \eqref{Coeffs-Rlt-all}. Furthermore, one of the following two conditions are satisfied:
	\begin{enumerate}
		\item $(\tilde{\mu}_2-\mu_2)^2 < 8\beta_4\mu_1$, i.e. the entropy inequality;
		\item $\tilde{\mu}_2 = \mu_2\,,$
	\end{enumerate}
	then, there is a small $\eps_2 > 0$, depending only on the all coefficients, $s$ and $d$, such that if
	\begin{equation}\label{IC-Ener-glb}
	\begin{aligned}
	E^{in} : = \| u^{in} \|^{2}_{H^{s}} +J \| \tilde{Q}^{in} \|^{2}_{H^{s}} + L \| \nabla Q^{in} \|^{2}_{H^{s}} +a \| Q^{in} \|^{2}_{H^{s}} \leq \eps_2 \,,
	\end{aligned}
	\end{equation}
	the system \eqref{CIQS}-\eqref{initial date} admits a unique global solution $(u,Q)$ satisfying
	\begin{equation*}
		\begin{aligned}
			u, \dot{Q} \in L^\infty (\R^+; H^s (\R^d)) \,, \ Q \in L^\infty (\R^+, H^{s+1} (\R^d)) \,, \ \nabla u \in L^2 (\R^+; H^s (\R^d)) \,,
		\end{aligned}
	\end{equation*}
	and subjecting to the following inequality
	\begin{equation}\label{Ener-Glb}
	\begin{aligned}
	\sup_{t \geq 0}(\| u\|^{2}_{H^{s}}
	+\| \dot{Q}\|^{2}_{H^{s}}
	+ \| \nabla Q\|^{2}_{H^{s}}
	+ \| Q\|^{2}_{H^{s}})(t) + \int^{\infty}_{0}\|\nabla u\|^{2}_{H^{s}}dt
	\leq C_{1}E^{in}
	\end{aligned}
	\end{equation}
	for some $C_1 > 0$,  depending only on the all coefficients, $s$ and $d$.
\end{theorem}

\begin{remark}
	In fact, the small global existence of the solution to the system \eqref{CIQS}-\eqref{initial date} has been initially proved by F. De Anna and A. Zarnescu in \cite{DeAnna-Zarnescu-JDE-2018}. However, they further required the following coefficients' constraints:
	\begin{equation}\label{AZ-Asump}
	\begin{aligned}
	\beta_5 + \beta_6 = 0 \,, \ \tilde{\mu}_2 = - \mu_2 \ (\textrm{or} \ \tilde{\mu}_2 = \mu_2 = 0) \,, \ J < J_0 \,, \ \mu_1 > \bar{\mu}_1 \,, \ \beta_4 > \tilde{C}_d
	\end{aligned}
	\end{equation}
	for some computable positive constant $J_0 = J_0 (\mu_1, a, b, c) > 0$, $\bar{\mu}_1 = \bar{\mu}_1 (a, b , c) > 0$ and $\tilde{C}_d = \tilde{C}_d (\tilde{\mu}_2, \beta_5, \beta_6, \mu_2) > 0$. In Theorem \ref{Main-Thm-2} here, the corresponding coefficients' relations are
	\begin{equation}\label{LM-Asump}
	\begin{aligned}
	\beta_5 + \beta_6 \in \R \,, \ J > 0 \,, \ \tilde{\mu}_2, \mu_2 \in \R \,, \ 8 \beta_4 \mu_1 > (\tilde{\mu}_2 - \mu_2)^2 \,.
	\end{aligned}
	\end{equation}
	Namely, when $\tilde{\mu}_2 \neq \mu_2$, the positive lower bound $\beta_4 \mu_1 > \tfrac{1}{8} ( \tilde{\mu}_2 - \mu_2 )^2 > 0$ is required. While $\tilde{\mu}_2 = \mu_2$, the positive lower bounds of the viscosities $\beta_4$ and $\mu_1$ in \eqref{AZ-Asump} are not required.
\end{remark}

We also can derive the energy decay estimate of the system \eqref{CIQS} on the periodic domain $\T^d$. More precisely,

\begin{theorem}[Decay estimate on torus $\T^d$]\label{Main-Thm-Decay}
	Consider the system \eqref{CIQS} on $(t,x) \in \R^+ \times \T^d$. Under the same assumptions in Theorem \ref{Main-Thm-2}, if
	\begin{equation}\label{IC-Avar}
	\begin{aligned}
	\int_{\T^d} u^{in} d x = 0
	\end{aligned}
	\end{equation}
	are further assumed, then there are constants $C_2, C_3 > 0$, depending only on $s$, $d$, all coefficients and initial data, such that the solution $(u, Q)$ constructed in Theorem \ref{Main-Thm-2} subjects to the decay estimate
	\begin{equation}\label{Decay-Bnd}
	\begin{aligned}
	\| u (t) \|^2_{H^s} + \| Q (t) \|^2_{H^{s+1}} + \| \dot{Q} (t) \|^2_{H^s} \leq C_2 E^{in} e^{- C_3 t}  \ ( \, \forall \, t \geq 0 \, ) \,.
	\end{aligned}
	\end{equation}
\end{theorem}

\subsection{Previous results and the novelty of this paper}

The analytical study of the inertial Qian-Sheng model of liquid crystal flow started only very recently. The second order material derivative in the inertial model brought tremendous difficulties, so its analytical study is much harder than the corresponding non-inertial case. The first result in this direction was obtained by De Anna and Zarnescu \cite{DeAnna-Zarnescu-JDE-2018}. They proved the global well-posedness under the assumptions that the initial data are small and the coefficients satisfy some further damping property. However, only the case $\beta_5+\beta_6 =0$ was considered. As a consequence, only small data local and global in time could be obtained. In \cite{DeAnna-Zarnescu-JDE-2018}, they also provided an example of twist-wave solutions, which are solutions of the coupled system for which the flow vanishes for all times. Furthermore, for the inviscid version of the inertial Qian-Sheng model, in \cite{Feireisl-Rocca-Schimperna-Zarnescu-2018-JHDE}, Feireisl et al. proved a global existence of the dissipative solution which is inspired from that of incompressible Euler equation defined by P.-L. Lions \cite{Lion-1996-BOOK}. We also mention the recent work of the third named author of this paper on the well-posedness of non-inertial Qian-Sheng model \cite{Ma-2020DCDS}. Comparing to other $Q$-tensor model, such as Beris-Edwards system, one of the main difficulty of Qian-Sheng model is it is hard to find Lyapunov functional, and the energy estimate does not close in $L^2$ sense. Moreover,  the second and the third named authors of this paper proved the same results of the global well-posedness corresponding to Theorem 1.3 in \cite{Luo-Ma-2021-DCDS}, where they started from the low Mach number limits of the compressible inertial Qian-Sheng model for nematic liquid crystals.

There are many works on the $Q$-tensor model without inertial effect but being different with the Qian-Sheng model. For instance, Paicu and Zarnescu \cite{Paicu-Zarnescu-2012-ARMA} proved the global weak solutions of the non-inertial Beris-Edwards system, a $Q$-tensor model. For more $Q$-tensor models of liquid crystals, readers can be referred to the works \cite{ADL-SIMA2014, DeAnna-Zarnescu-CMS-2016, Kirr-Wilkinson-Zarnescu-2014-JSP, Paicu-Zarnescu-2011-SIAM, Xiao-2017-JDE, Zarnescu-TMMA-2012} and the related references therein, for instance.

For the assumption $\beta_5 + \beta_6 = 0$ in \eqref{Entropy-Coeffs-1}, there is a more specialized form $\beta_5 = \beta_6 = 0$ in the physics literatures \cite{PopaNita-Oswald-2002-PRE,PopaNita-Sluckin-Kralj-2005-PRE}. Very recently, Li and Wang \cite{Li-Wang-2019-arXiv} justified rigorously the limit from incompressible inertial Qian-Sheng model to the Ericksen-Leslie theory by employing the Hilbert method, which included the case $\beta_5 + \beta_6 \neq 0$. However,  an extra term $\beta_7 (AQ^2 + Q^2 A)$ was added to treat this case. In the current paper, the case $\beta_5 + \beta_6 \neq 0$ is also included. Our analysis indicates that the $\beta_7$ term is not needed.

The main novelty of this paper is, the roles played by the entropy inequality, Condition (H) in the well-posedness of the Qian-Sheng model \eqref{CIQS} are illustrated. For local in time well-posedness, we make clarification that the large and small data are quite different, which is directly based on entropy inequality and energy dissipation: to obtain large data local in time well-posedness, we require both the entropy inequality and the Condition (H). If any one of these two conditions is not satisfied, we can only obtain small data local existence. When either the entropy inequality, or $\tilde{\mu}_2 = \mu_2$ is satisfied, the local small local in time solutions can be extended globally. Furthermore, in both global and local in time cases, the range of our assumptions on the coefficients are significantly enlarged, comparing to previous results.

We finally mention that one of the main motivation of Qian-Sheng model \cite{Qian-Sheng-PRE-1998} was to provide a $Q$-tensor version of the classical Ericksen-Leslie system. For the inertial Ericksen-Leslie system, there has been recently some progress in the context of smooth solutions, see \cite{GJLLT-2019, HJLZ-1, HJLZ-2, JL-SIAMJMA-2019, JLT-M3AS-2019}, in which the basic energy dissipation law (Lyapunov functional) is equivalent to entropy inequality. It will be an interesting question to relate the entropy inequalities and energy dissipation in the Qian-Sheng and Ericksen-Leslie models. This will be a future work.

%%%%%%%%%%%%%%%%%%%%%%%%%%%%%%%%%%%%%%%%%%%%%%%%%%%%%%%%%%%%%%%%%%%%%%%%%%%%%%%%%%%%%%%%%%%%%%

\subsection{Main ideas and sketch of proofs}

One of the important assumptions in the above theorems is $a > 0$. This captures a regime of physical interest but not the most interesting physical regime (which would be for $a \leq 0$, meaning the ``deep nematic" regime, see \cite{Mottram-2014-arXiv}). Technically the assumption $a > 0$ provides an additional damping effect $a Q$, which play an essential role in deriving the global energy estimates.

From the $L^2\mbox{-}$estimate (not closed) \eqref{L2 sums} in Section \ref{Sec: Apriori}, we know that
\begin{align*}
	& \tfrac{1}{2} \tfrac{d}{d t} \big( \| u \|^2_{L^2} + J \| \dot{Q} \|^2_{L^2} + L \| \nabla Q \|^2_{L^2} + a \| Q \|^2_{L^2} \big)
	+ \beta_1 \| Q : A \|^2_{L^2} \\
	&+ \int_{\R^d} F ( \nabla u , \dot{Q} , [\Omega, Q] ) d x + (\beta_5 + \beta_6 ) \langle AQ, A \rangle \leq \textrm{ some other terms} \,,
\end{align*}
which has corresponding version of higher order derivatives' estimates. It is easy to see that the Condition (H) on the viscosities $\beta_4, \mu_1 > 0, \tilde{\mu}_2, \mu_2 \in \R$ given in \eqref{Condition-H} exactly supplies the {\em coercivity} of dissipative functional $F(\cdot, \cdot, \cdot)$. More precisely, the effective dissipative structures resulted from $\int_{\R^d} F ( \nabla u , \dot{Q} , [\Omega, Q] ) d x$ are $\| \nabla u \|^2_{L^2}$ and $ \| \dot{Q} \|^2_{L^2} $.

If $\beta_5 + \beta_6 \neq 0$, the quantity $(\beta_5 + \beta_6) \langle A Q , A \rangle$ can only be bounded by the quantity $\| Q \|_{H^s} \| \nabla u \|^2_{H^s}$, which can only be absorbed by the dissipation $\| \nabla u \|^2_{H^s}$ when the energy $\| Q \|_{H^s}$ is sufficiently small. Consequently, as shown in Theorem \ref{Main-Thm-1-prime}, we will merely construct a unique local solution with small energy for the case $\beta_5 + \beta_6 \neq 0$. We emphasize that the structure $ ( \beta_5 + \beta_6 ) \langle A Q , A \rangle $ will not result to any difficulty when we prove the global well-posedness with small initial data. Oppositely, if the term $( \beta_5 + \beta_6 ) \langle A Q, A \rangle$ vanishes, namely, $\beta_5 + \beta_6 = 0$, we can construct the local solution with large initial data and the global solution with small initial data provided that the viscosities $\beta_4, \mu_1 > 0, \tilde{\mu}_2, \mu_2 \in \R$ satisfy the Condition (H).

Moreover, even the viscosities $\beta_4, \mu_1 > 0, \tilde{\mu}_2, \mu_2 \in \R$ does not satisfy the Condition (H) under $\beta_5 + \beta_6 =0$ and $\tilde{\mu}_2 = \mu_2 \neq 0$, we can also prove the local well-posedness under small initial data. However, proving the large local solution fails in that case. One notices that $\beta_5 + \beta_6 = 0$ and $\tilde{\mu}_2 - \mu_2 \neq 0$ exactly obey the entropy inequality
\begin{equation}\label{Entropy4}
\begin{aligned}
\int_{\R^d} \! ( \sigma'_{ij} A_{ij} - g'_{\alpha \beta} \mathscr{N}_{\alpha \beta} ) d x \geq 0 \,,
\end{aligned}
\end{equation}
derived from the relation \eqref{Entropy3}, which is often regarded as a physical condition. But why does not we construct the large local solution under the coefficients constraints $\beta_5 + \beta_6 = 0$ and $\tilde{\mu}_2 - \mu_2 \neq 0$ without the Condition (H)? It results from the asymmetric structure $\beta_5 Q_{jl} A_{li} + \beta_6 Q_{il} A_{lj}$ given in the viscous stress $\sigma'_{ij}$, i.e., \eqref{sigma-prime}. One notices that
\begin{equation}\label{Decomp-beta5+beta6}
\begin{aligned}
\beta_5 Q_{jl} A_{li} + \beta_6 Q_{il} A_{lj} = & \tfrac{1}{2} ( \beta_5 + \beta_6 ) ( Q_{jl} A_{li} + Q_{il} A_{lj} ) - \tfrac{1}{2} (\beta_6 - \beta_5) ( Q_{jl} A_{li} - Q_{il} A_{lj} ) \\
= & \underbrace{ \tfrac{1}{2} ( \beta_5 + \beta_6 ) ( Q_{jl} A_{li} + Q_{il} A_{lj} ) }_{\textrm{ Symmetric part }} \ \underbrace{ - \tfrac{1}{2} \mu_2 ( Q_{jl} A_{li} - Q_{il} A_{lj} ) }_{ \textrm{ Skew-symmetric part } } \,,
\end{aligned}
\end{equation}
where the last equality is implied by the Parodi's relation \eqref{Parodi-Rlt}. Under $\beta_5 + \beta_6 = 0$, the symmetric part will automatically be zero. Furthermore, in the inequality \eqref{Entropy4}, $\sigma_{ij}$ is multiplied by the {\em symmetric tensor} $A_{ij}$, which is such that the skew-symmetric part of $\beta_5 Q_{jl} A_{li} + \beta_6 Q_{il} A_{lj}$ vanishes. Consequently, the term $\beta_5 Q_{jl} A_{li} + \beta_6 Q_{il} A_{lj}$ will not affect the positivity of whole entropy inequality \eqref{Entropy4}. However, we multiply by $u$ in the $u$-equation of \eqref{CIQS} and integrate over $x \in \R^d$, which tells us that the term $\beta_5 Q_{jl} A_{li} + \beta_6 Q_{il} A_{lj}$ actually multiplies by the {\em symmetric tensor} $\nabla u$ after the integration by parts. Then, the skew-symmetric part in \eqref{Decomp-beta5+beta6} will not vanish, which results to a quantity $ \tfrac{1}{2} \mu_2 \langle [\Omega, Q] , \nabla u \rangle $. This quantity is contained in the expression of $ \int_{\R^d} F(\nabla u, \dot{Q}, [\Omega, Q]) d x $. Recall that
\begin{equation*}
	\begin{aligned}
		F(\nabla u, \dot{Q}, [\Omega, Q]) = \tfrac{1}{2} \beta_4 |\nabla u|^2 + \mu_1 ( |\dot{Q}|^2 + |[\Omega, Q]|^2 ) \\
		- \tfrac{1}{2} (\tilde{\mu}_2 - \mu_2) \nabla u : \dot{Q}- \mu_2 \nabla u : [\Omega, Q]  \,.
	\end{aligned}
\end{equation*}
If $\tilde{\mu}_2 = \mu_2 \neq 0$ and $\beta_4, \mu_1 > 0, \tilde{\mu}_2, \mu_2 \in \R$ do not obey the Condition (H), there will a unsigned quantity $- \mu_2 \langle \nabla u, [\Omega, Q] \rangle$ in $F (\nabla u, \dot{Q}, [\Omega, Q])$, which break down the coercivity of $F$. On the other hand, the quantity $- \mu_2 \langle \nabla u, [\Omega, Q] \rangle$ will be bounded by $\| Q \|_{H^s} \| \nabla u \|^2_{H^s}$, which can be absorbed by the dissipation $\| \nabla u \|^2_{H^s}$ only provided that the norm $\| Q \|_{H^s}$ is sufficiently small. In this sense, we can only prove the small local solution in the case $\tilde{\mu}_2 = \mu_2 \neq 0$ and $\beta_5 + \beta_6 = 0$. Of course, the quantity $- \mu_2 \langle \nabla u, [\Omega, Q] \rangle$ is easy to be dealt with when prove the global well-posedness with small initial data, so that this small local solution can be globally extended.

When we derive the a priori energy estimates, the basic cancellation
\begin{equation*}
	\begin{aligned}
		\langle  \div (-\nabla Q\odot\nabla Q) ,  u  \rangle + \langle \Delta Q, u \cdot \nabla Q \rangle=0
	\end{aligned}
\end{equation*}
is crucial. Based on the a priori estimate (see Lemma \ref{priori} below), we employ the the mollifier scheme and the continuity arguments to prove the all kinds of local results in Theorem \ref{Main-Thm-1} and Theorem \ref{Main-Thm-1-prime}. In order to globally extend the local solutions constructed in Theorem \ref{Main-Thm-1}- \ref{Main-Thm-1-prime}, we need seek more dissipative structures  $\| \nabla Q \|^2_{H^s}$ and $\| Q \|^2_{H^s}$. Then, with sufficiently small initial data, we extend globally the local solutions constructed above by applying the continuity arguments.

%%%%%%%%%%%%%%%%%%%%%%%%%%%%%%%%%%%%%%%%%%%%%%%%%%%%%%%%%%%%%%%%%%%%%%%%%%%%%%%%%%%%%%%%%%%%%%

\subsection{Organizations of this paper}

This paper is organized as follows: in the next section, we derive the a priori estimates of the incompressible inertial Qian-Sheng model. Based on the a priori estimates, we prove the local well-posedness (including the large initial solution under the Condition (H) and the small initial data solution when the Condition (H) failed) by employing the mollifier method in Section \ref{Sec: Local}. Finally, in Section \ref{Sec: Global}, we prove the global classical solution and verify the time decay estimate for the system \eqref{CIQS} considered on the torus $\T^d$.

\section{A priori estimates}\label{Sec: Apriori}

In this section, we aim at deriving the a priori estimate of the system \eqref{CIQS}-\eqref{initial date} associated with the $H^s$-norms for $s > \tfrac{d}{2} + 1$.

\subsection{Preliminaries} We first give a known lemma, which will be frequently used in the current paper.

\begin{lemma}[Moser-type inequality, \cite{Majda-Bertozzi-2002-BOOK}]\label{Lemma3}
	For functions $ f , g \in H^s \bigcap L^{\infty}, s \in \mathbb{Z}_{+}\bigcup \{0\}$, then for any multi-index $ \alpha  = ( \alpha_1,\alpha_2,...,\alpha_d ) \in \mathbb{N} ^d$ and
	$ 1 \leq | \alpha | \leq s $, we have
	\begin{equation}
	\begin{aligned}
	\| \partial ^ {\alpha} ( fg) \|_{L^2} \lesssim & \| f \|_{L^ \infty} \| g \|_{\dot{H}^s} + \| f \|_{\dot{H}^s } \| g \|_{L^ \infty} \,, \\
	\| [\partial^{\alpha}, f] g \|_{L^{2}}
	\lesssim & ( \| \nabla f \|_{L^{\infty}} \| g \|_{H^{s-1}}
	+ \|  f \|_{\dot{H}^s} \| g \|_{L^{\infty}} ) \,.
	\end{aligned}
	\end{equation}
	In particular, if $ s > \tfrac{d}{2}$, then  $f,g \in H^s $, we have
	\begin{equation}
	\begin{aligned}
	\| f g \|_{H^s} \lesssim \|f\|_{H^s} \|g\|_{H^s} \,.
	\end{aligned}
	\end{equation}
\end{lemma}

\subsection{A priori estimate.}\label{ape}
In this subsection, we derive the a priori estimate of the system \eqref{CIQS}-\eqref{initial date}. For convenience, we first introduce the following energy dissipation $D (t)$:
\begin{equation}\label{Local-Dissipation}
\begin{aligned}
D(t) = \| \nabla u \|^{2}_{H^{s}} + \| \dot{Q} \|^2_{H^s} + \beta_1 \sum_{|k|=0}^s \| \partial^k A : Q \|^2_{L^2} \,.
\end{aligned}
\end{equation}
Then we will give the following lemma about the a priori estimate of the system \eqref{CIQS}-\eqref{initial date}.

\begin{lemma}\label{priori}
	Let integer $ s > \tfrac{d}{2} + 1 $ ($d=2,3$). Assume $(u, Q ) $ is a sufficiently smooth solution to system \eqref{CIQS}-\eqref{initial date} on some interval [0,T]. Then there exists a constant $ c_0 > 0 $, depending only on the all coefficients, $s$ and $d$, such that for all $ t \in [ 0 , T ]$,
	\begin{equation}\label{Ener-Loc}
	\begin{aligned}
	\tfrac{1}{2} \tfrac{d}{d t} E (t) + c_0 D (t) \lesssim & ( \delta_H + |\beta_5 + \beta_6| ) E^\frac{1}{2} (t) D (t)     \\
	& + (1+E^{\frac{1}{2}}(t)) E^{\frac{1}{2}}(t)D^{\frac{1}{2}}(t)\mathcal{A}^{\frac{1}{2}}(t) + \delta_H |\tilde{\mu}_2 - \mu_2|^2 E (t) \,,
	\end{aligned}
	\end{equation}
	where $E (t)$ is defined in \eqref{Ener} and the functional $\mathcal{A}(t)$ is defined as
	\begin{equation}\label{A}
	\begin{aligned}
	\mathcal{A}(t) = \|u\|^2_{\dot{H}^s} +\|\dot{Q}\|^2_{H^s}+ \|\nabla Q\|^2_{H^s} + \|Q\|^2_{H^ s} \,,
	\end{aligned}
	\end{equation}
	and
	\begin{equation*}
		\delta_H =
		\left\{
		\begin{array}{l}
			0 \,, \ \textrm{ if the Condition (H) holds} \,, \\
			1 \,, \ \textrm{ if the Condition (H) does not hold} \,.
		\end{array}
		\right.
	\end{equation*}
\end{lemma}

\begin{proof} We divide the proof into two steps.
	
	\vspace*{2mm}
	
	{\em Step 1. $L^{2}$-estimate.}
	
	\vspace*{2mm}
	
	First, taking $ L^{2} $-inner product with $ u $ in the first equation of system \eqref{CIQS}, we obtain
	\begin{equation}\label{L2 2}
	\begin{aligned}
	\tfrac{1}{2} \tfrac{d}{d t} \| u \|^2_{L^2} + \tfrac{\beta_{4}}{2} \| \nabla u \|^2_{L^{2}}
	=  \langle \div \Sigma_{1} , u \rangle + \langle \div \Sigma_{2} , u \rangle + \langle \div \Sigma_{3} , u \rangle\,.
	\end{aligned}
	\end{equation}
	Second, taking $ L^{2} $-inner product with $ \dot{Q} $ in the third equation of system \eqref{CIQS}, we can deduce that
	\begin{equation}\label{L2 3}
	\begin{aligned}
	&  \tfrac{1}{2} \tfrac{d}{d t} \big( J \| \dot{Q} \|^{2}_{L^2} + L \| \nabla Q \|^{2}_{L^2} + a \| Q \|^{2}_{L^2} \big) + \mu_{1} \| \dot{Q} \|^2_{L^{2}} \\
	= & L \langle \Delta Q , u \cdot \nabla Q \rangle  + b \langle ( Q^{2} - \tfrac{1}{d} | Q |^{2} I_{d} ) , \dot{Q} \rangle  - c \langle Q | Q |^{2} , \dot{Q} \rangle \\
	&+ \tfrac{\tilde{\mu}_{2}}{2} \langle A , \dot{Q} \rangle + \mu_{1} \langle [ \Omega , Q ] , \dot{Q} \rangle \,.
	\end{aligned}
	\end{equation}
	Therefore, combining the above equalities \eqref{L2 2} and \eqref{L2 3}, we know
	\begin{align}\label{L2 sum}
		\no & \tfrac{1}{2} \tfrac{d}{dt} \big(  \| u \|^2_{L^2} + J \| \dot{Q} \|^2_{L^2} + L \| \nabla Q \|^2_{L^2} + a \| Q \|^{2}_{L^2} \big) + \tfrac{\beta_{4}}{2}  \|\nabla u \|^2_{L^{2}} + \mu_{1}  \| \dot{Q} \|^2_{L^{2}} \\
		\no = & \underbrace{  \langle  \div \Sigma_{1} ,  u  \rangle +L \langle \Delta Q, u \cdot \nabla Q \rangle }_{H_1} \\
		\no &  +\langle \div\Sigma_{2}, u \rangle + \tfrac{\mu_{2}}{2} \langle \div \dot{Q} , u \rangle -\tfrac{\mu_{2}}{2} \langle \div ([\Omega , Q ]) , u \rangle  \\
		\no& \underbrace{\ \ \ \ \ \ \ \ \ \ \ \ \ \ \ \ \  -\mu_{1} \langle \div ([Q ,[\Omega , Q ]]) , u \rangle  + \tfrac{\tilde{\mu}_{2}}{2} \langle A   , \dot{Q} \rangle \ }_{H_2}\\
		&\underbrace{+b \langle (Q^{2} - \tfrac{1}{d} | Q |^{2} I_{d}), \dot{Q} \rangle - c \langle Q | Q |^{2},\dot{Q}\rangle
			+ \mu_{1} \langle \div \big([Q , \dot{Q}] \big) , u \rangle + \mu_{1} \langle [\Omega,Q], \dot{Q} \rangle }_{H_3}\,.
	\end{align}
	
	Now we compute each term on the right-hand side of the equality \eqref{L2 sum}.
	Simple calculation tells us that
	\begin{equation}\nonumber
	\begin{aligned}
	H_{1} = -\langle  \div (\nabla Q\odot\nabla Q) ,  u  \rangle
	+ \langle \Delta Q, u \cdot \nabla Q \rangle=0     \,.
	\end{aligned}
	\end{equation}
	For calculating the second part $H_2$, we will repeatedly use that $\nabla u=A+\Omega$ and the fact $\tr(BC)=0$ for any symmetric matrix $B$ and skew-adjoint matrix $C$ (i.e., $ C_{ij}=-C_{ji}$). Recalling the definition of $\Sigma_{3}$ defined in \eqref{Sigma1-4}, we have
	\begin{equation}\nonumber
	\begin{aligned}
	H_{2} = & - \beta_{1} \langle Q \tr(QA) , \nabla u \rangle -\beta_{5} \langle A Q , \nabla u \rangle -\beta_{6} \langle Q A,  \nabla u \rangle + \tfrac{\mu_{2}}{2} \langle [ \Omega , Q ],  \nabla u \rangle \\
	& - \tfrac{\mu_{2}}{2} \langle \dot{Q} , \nabla u \rangle + \tfrac{\tilde{\mu}_{2}}{2} \langle A  , \dot{Q} \rangle + \mu_{1} \langle [Q, [ \Omega , Q ] ], \nabla u \rangle \,.
	\end{aligned}
	\end{equation}
	We begin with
	\begin{equation}\nonumber
	\begin{aligned}
	-\beta_{1} \langle  Q \tr(QA) , \nabla u \rangle
	=&
	-\beta_{1} \int_{\mathbb{R}^d} Q_{ij} Q_{lk} A_{kl} ( A_{ij} + \Omega_{ij}) dx\\
	=&
	-\beta_{1} \int_{\mathbb{R}^d} Q_{ij} Q_{lk} A_{kl} A_{ij} dx
	-\beta_{1} \int_{\mathbb{R}^d} Q_{ij} Q_{lk} A_{kl} \Omega_{ij} dx\\
	=&
	-\beta_{1} \int_{\mathbb{R}^d} (Q : A )^{2} dx
	-\beta_{1} \int_{\mathbb{R}^d} (Q : \Omega ) (Q : A ) dx\\
	=&
	-\beta_{1} \|  Q : A  \|^{2}_{L^{2}} \, ,
	\end{aligned}
	\end{equation}
	where we use of the relation $Q:\Omega=0$. Moreover, we have
	\begin{align*}
		&-\beta_{5} \langle A Q , \nabla u \rangle
		-\beta_{6} \langle Q A,  \nabla u \rangle
		+\tfrac{\mu_{2}}{2} \langle [ \Omega , Q ],  \nabla u \rangle\\
		=&
		-\beta_{5} \int_{\mathbb{R}^d} \tr\{ Q A (A + \Omega)\}dx
		-\beta_{6} \int_{\mathbb{R}^d} \tr\{ A Q (A + \Omega) \}dx \\
		&
		+\tfrac{\mu_{2}}{2}\int_{\mathbb{R}^d} \tr\{[\Omega,Q](A+\Omega)\}dx
	\end{align*}
	\begin{align*}
		=&
		-\beta_{5}\int_{\mathbb{R}^d} \tr\{( Q A + A Q )( A + \Omega )\}dx
		-(\beta_{6}-\beta_{5})\int_{\mathbb{R}^d} \tr\{ A Q ( A + \Omega )\}dx\\
		&
		+\tfrac{\mu_{2}}{2}  \int_{\mathbb{R}^d} \tr\{[ \Omega , Q ]( A + \Omega )\}dx\\
		=&
		-\beta_{5}\int_{\mathbb{R}^d} \tr\{( Q A + A Q ) A \} dx
		-(\beta_{6}-\beta_{5}) \int_{\mathbb{R}^d} \tr\{AQ(A+\Omega)\} dx\\
		&
		+\tfrac{\mu_{2}}{2} \int_{\mathbb{R}^d} \tr\{[\Omega , Q]( A + \Omega )\} dx\\
		=&
		-(\beta_{5} + \beta_{6}) \int_{\mathbb{R}^d} \tr( A Q A ) dx
		-(\beta_{6}- \beta_{5}) \int_{\mathbb{R}^d} \tr( A Q \Omega ) dx
		+\tfrac{\mu_{2}}{2}  \int_{\mathbb{R}^d} \tr( A [ \Omega , Q ] ) dx\\
		=&
		-(\beta_{5} + \beta_{6}) \langle  A Q, A \rangle   +
		\mu_{2} \langle  A , [ \Omega , Q ]  \rangle \,,
	\end{align*}
	where we use the Parodi's ration $\beta_{6}-\beta_{5}=\mu_{2}$ and
	\begin{equation}\nonumber
	\begin{aligned}
	\int_{\mathbb{R}^d} \tr( A Q \Omega ) dx =
	- \int_{\mathbb{R}^d} \tr( A \Omega Q) dx\,.
	\end{aligned}
	\end{equation}	
	Furthermore, from the facts $\tr\dot{Q}=0$ and $\nabla u =A+\Omega$, we deduce that
	\begin{equation}
	\begin{aligned}
	-\tfrac{\mu_{2}}{2} \langle \dot{Q} ,\nabla u \rangle
	+\tfrac{\tilde{\mu}_{2}}{2} \langle A  ,\dot{Q} \rangle
	=-\tfrac{\mu_{2}}{2} \langle \dot{Q} , A \rangle
	+\tfrac{\tilde{\mu}_{2}}{2} \langle  A ,\dot{Q} \rangle
	= \tfrac{1}{2} ( \tilde{\mu}_{2} -  \mu_{2})
	\langle  A ,\dot{Q} \rangle \,.
	\end{aligned}
	\end{equation}
	Finally, we calculate that
	\begin{equation}\nonumber
	\begin{aligned}
	\mu_{1} \langle [Q, [ \Omega , Q ] ], \nabla u \rangle
	=&
	\mu_{1} \int_{\mathbb{R}^d} \tr\{( [\Omega , Q ] Q - Q [ \Omega, Q])(A + \Omega)\} dx\\
	=&
	\mu_{1} \int_{\mathbb{R}^d} ([ \Omega , Q]_{ik} Q_{kj} - Q_{ik} [ \Omega , Q]_{kj}) \Omega_{ji} dx\\
	=&
	\mu_{1} \int_{\mathbb{R}^d} (Q_{kj} \Omega_{ji} [ \Omega, Q ]_{ik} - \Omega_{kj} Q_{ji} [\Omega, Q]_{ik} ) dx\\
	=&
	-\mu_{1} \| [\Omega , Q ]  \|^2_{L^{2}} \,.
	\end{aligned}
	\end{equation}
	Combining the all above estimates reduces to
	\begin{equation}\nonumber
	\begin{aligned}
	H_{2} -\mu_{1} \| \dot{Q} \|^{2}_{L^{2}} - & \tfrac{\beta_{4}}{2}\| \nabla u \|^2_{L^2} \\
	= &
	-\beta_{1} \| Q : A \|^{2}_{L^{2}}
	- \int_{\R^d} F ( \nabla u, \dot{Q} , [\Omega, Q] ) d x - ( \beta_{5} + \beta_{6} ) \langle A Q, A \rangle
	\,,
	\end{aligned}
	\end{equation}
	where the functional $F( \cdot, \cdot, \cdot )$ is defined in \eqref{Condition-H}. For the term $H_{3}$, the H\"older inequality and Sobolev embedding imply that
	\begin{equation}\nonumber
	\begin{aligned}
	H_{3}
	\lesssim &
	\|\nabla Q\|_{L^{2}}\| Q\|_{H^1}\|\dot{Q}\|_{L^{2}}
	+\|\nabla Q\|^{3}_{L^{2}}\|\dot{Q}\|_{L^{2}}
	+ \| Q \|_{H^2} \| \dot{Q} \|_{H^2}  \| \nabla u \|_{L^{2}} \\
	\lesssim &
	(1+E^{\frac{1}{2}}(t))E^{\frac{1}{2}}(t)D^{\frac{1}{2}}(t)(\|\nabla Q\|_{L^2}+\|Q\|_{H^2})\,.
	\end{aligned}
	\end{equation}
	Summarizing all the previous estimates, we get
	\begin{equation}\label{L2 sums}
	\begin{aligned}
	& \tfrac{1}{2} \tfrac{d}{d t} \big( \| u \|^2_{L^2} + J\| \dot{Q} \|^2_{L^2} + L \| \nabla Q \|^2_{L^2} + a \| Q \|^2_{L^2} \big)  + \beta_{1} \| Q : A \|^{2}_{L^{2}} \\
	&+ \int_{\R^d} F ( \nabla u, \dot{Q} , [\Omega, Q] ) d x + ( \beta_{5} + \beta_{6} ) \langle A Q, A \rangle \\
	\lesssim &
	(1+E^{\frac{1}{2}}(t))E^{\frac{1}{2}}(t)D^{\frac{1}{2}}(t)(\|\nabla Q\|_{L^2}+\|Q\|_{H^2}) \,.
	\end{aligned}
	\end{equation}	
	
	\vspace*{4mm}
	
	{\em Step 2. Higher order derivative estimates.}
	
	\vspace*{4mm}
	For all multi-index $k \in \mathbb{N}^d$ with $1\leq|k|\leq s(s>\frac{d}{2}+1)$, from acting $\partial^{k}$
	on the first equation of \eqref{CIQS} and taking $L^{2}$-inner product with $ \partial^{k} u $, we deduce that
	\begin{equation}\label{priori u}
	\begin{aligned}
	&\tfrac{1}{2} \tfrac{d}{dt} \| \partial^{k} u \|^{2}_{L^2} + \langle  \partial^{k}( u \cdot \nabla u ) , \partial^{k} u \rangle + \tfrac{\beta_{4}}{2}  \|\nabla \partial^{k} u \|_{L^{2}}^{2}\\
	=&
	\langle \div \partial^{k} \Sigma_{1}, \partial^{k} u \rangle
	+ \langle \div \partial^{k} \Sigma_{2}, \partial^{k} u \rangle
	+ \langle \div \partial^{k} \Sigma_{3}, \partial^{k} u \rangle\,.
	\end{aligned}
	\end{equation}
	
	Then, applying multi-derivative operator $\partial^{k}$ on the third equation of \eqref{CIQS}, and taking $L^{2}$-inner product with $\partial^{k}\dot{Q}$, one easily yields that
	\begin{equation}\label{priori Q}
	\begin{aligned}
	&\tfrac{1}{2} \tfrac{d}{d t} \big( J \| \partial^{k} \dot{Q} \|^{2}_{L^2}
	+ L \| \nabla  \partial^{k} Q \|^{2}_{L^2}
	+ a \| \partial^{k} Q \|^{2}_{L^2} \big)
	+ \mu_{1}  \| \partial^{k} \dot{Q} \|^{2}_{L^{2}}\\
	=&
	-J \langle \partial^{k}( u \cdot \nabla  \dot{Q}) , \partial^{k} \dot{Q} \rangle
	+L \langle  \Delta \partial^{k} Q , \partial^{k} ( u \cdot \nabla Q) \rangle
	-a \langle \partial^{k} Q, \partial^{k} ( u \cdot \nabla  Q ) \rangle\\
	&
	+ b \langle \partial^{k}  (Q^{2}- \tfrac{1}{d} | Q |^{2}I_{d}) ,  \partial^{k} \dot{Q} \rangle
	-c \langle \partial^{k} (Q | Q |^{2} ) ,  \partial^{k} \dot{Q} \rangle
	+ \tfrac{\tilde{\mu}_{2}}{2} \langle \partial^{k} A ,
	\partial^{k} \dot{Q} \rangle\\
	&
	+\mu_{1} \langle \partial^{k} ([ \Omega , Q ] ), \partial^{k} \dot{Q} \rangle
	\,.
	\end{aligned}
	\end{equation}
	Therefore, for all $1\leq|k|\leq s$, it is easily derived from combining the above equalities \eqref{priori u} and \eqref{priori Q} that
	\begin{equation}\label{priori hk}
	\begin{aligned}
	& \tfrac{1}{2} \tfrac{d}{dt} \big(
	\| \partial^{k} u \|^{2}_{L^2} + J \| \partial^{k} \dot{Q} \|^{2}_{L^2}
	+L \| \nabla \partial^{k} Q \|^{2}_{L^2}
	+ a \| \partial^{k} Q \|^{2}_{L^2} \big) \\
	&
	+ \tfrac{\beta_{4}}{2}  \| \nabla \partial^{k} u \|^{2}_{L^{2}}
	+ \mu_{1}  \| \partial^{k} \dot{Q} \|^{2}_{L^{2}}\\
	=&
	\underbrace{ L \langle  \Delta \partial^{k} Q , \partial^{k}( u \cdot \nabla Q) \rangle
		+  \langle  \div \partial^{k} \Sigma_{1} , \partial^{k} u  \rangle   }_{I_1} \\
	&
	\underbrace{  - \langle  \partial^{k}( u \cdot \nabla u ) , \partial^{k} u \rangle
		-J \langle \partial^{k}( u \cdot \nabla  \dot{Q}) , \partial^{k} \dot{Q} \rangle
		-a \langle \partial^{k} Q, \partial^{k} ( u \cdot \nabla  Q ) \rangle}_{I_{2}} \\
	& \left.
	\begin{array}{l}
	+ \langle \div \partial^{k} \Sigma_{2} , \partial^{k} u \rangle + \tfrac{\mu_{2}}{2} \langle \div \partial ^ k \dot{Q} , \partial ^ k u \rangle -\tfrac{\mu_{2}}{2} \langle \div \partial ^ k ([\Omega , Q ]) , \partial ^ k u \rangle \\[4mm]
	\ \ \ \ \ \ \ \ \ \ \ \ \ \ \ \ \ \ \ \ \ \ \  -\mu_{1} \langle \div \partial ^ k ([Q ,[\Omega , Q ]]) , \partial ^ k u \rangle
	+ \tfrac{\tilde{\mu}_{2}}{2} \langle \partial ^ k A  , \partial ^ k \dot{Q} \rangle
	\end{array}
	\right\} {I_3}\\
	&
	\underbrace{+ \mu_{1} \langle \div \partial ^ k ([Q , \dot{Q}]), \partial ^ k u \rangle +\mu_{1} \langle \partial^{k} ([ \Omega , Q ] ), \partial^{k} \dot{Q} \rangle}_{I_4} \\
	&
	\underbrace{+ b \langle \partial^{k}  (Q^{2}- \tfrac{1}{d} | Q |^{2}I_{d}) ,  \partial^{k} \dot{Q} \rangle}_{I_5}
	\underbrace{-c \langle \partial^{k} (Q | Q |^{2} ) ,  \partial^{k} \dot{Q} \rangle }_{I_6}
	\,.
	\end{aligned}
	\end{equation}
	We now turn to deal with $I_{i}(1\leq i\leq 6)$ term by term. It is obvious that the terms in $I_{1}$
	have a cancellation relation, so we can get by
	employing the H\"older inequality and the Sobolev embedding theory that
	\begin{equation}\label{I 1}
	\begin{aligned}
	I_{1}=&
	-\langle \partial_{j} \partial^{k} Q_{\alpha \beta} , \partial_{j} u_{i}  \partial_{i} \partial^{k} Q_{\alpha \beta} \rangle
	-\sum_{\substack{m_{1}+m_{2}=k,\\1 \leq | m_{2} | \leq |k|-1}}
	\langle  \Delta \partial^{m_{1}} Q_{\alpha \beta}   \partial_{i} \partial^{m_{2}} Q_{\alpha \beta}, \partial^{k} u_{i} \rangle\\
	&
	-\sum_{\substack{m_{1} + m_{2}= k,\\ 1 \leq | m_{2} | \leq |k|-1}}
	\langle \partial^{m_{1}} u_{i} \partial_{i} \partial_{j} \partial^{m_{2}} Q_{\alpha \beta}
	+  \partial_{j} \partial^{m_{1}} u_{i}  \partial_{i} \partial^{m_{2}} Q_{\alpha \beta},  \partial_{j} \partial^{k} Q_{\alpha \beta} \rangle\\
	&
	-\langle \Delta Q_{\alpha \beta}   \partial_{i} \partial^{k} Q_{\alpha \beta}, \partial^{k} u_{i} \rangle\\
	\lesssim&
	\| \nabla u \|_{L^{\infty}} \| \nabla \partial^{k} Q \|^{2}_{L^{2}}
	+\sum_{\substack{m_{1} + m_{2} = k, \\ 1 \leq | m_{2} | \leq |k|-1}}
	\| \Delta \partial^{m_{1}} Q \|_{L^{2}} \| \nabla \partial^{m_{2}} Q \|_{L^{4}} \| \partial^{k} u \|_{L^{4}}\\
	&
	+\sum_{\substack{m_{1} + m_{2}=k,\\ 1 \leq | m_{2} | \leq |k|-1}}
	(\| \partial^{m_{1}} u \|_{L^{\infty}} \| \Delta \partial^{m_{2}} Q \|_{L^{2}}
	+\|\nabla \partial^{m_{1}} u \|_{L^{4}}
	\| \nabla \partial^{m_{2}} Q \|_{L^{4}}  )\| \nabla \partial^{k} Q \|_{L^{2}}\\
	&
	+\| \Delta Q \|_{L^{4}} \| \partial^{k} u \|_{L^{4}} \| \nabla \partial^{k} Q \|_{L^{2}}
	\lesssim
	\| \nabla Q \|^{2}_{H^{s}} \| \nabla u\|_{H^{s}} \lesssim E^{\frac{1}{2}}(t)D^{\frac{1}{2}}(t)\|\nabla Q\|_{H^s} \,.
	\end{aligned}
	\end{equation}
	As to $I_{2}$, from using the H\"older inequality, the Sobolev embedding theory and the fact $\div u =0$ to deduce that
	\begin{align}
		\no I_{2}=&
		-\sum_{\substack{m_{1}+m_{2}=k,\\1 \leq | m_{1} |}}  \langle  \partial^{m_{1}} u \cdot \nabla \partial^{m_{2}} u , \partial^{k} u \rangle
		-J \sum_{\substack{m_{1}+m_{2}=k,\\1 \leq | m_{1} |}} \langle \partial^{m_{1}} u \cdot \nabla \partial^{m_{2}}\dot{Q} , \partial^{k} \dot{Q} \rangle
	\end{align}
	\begin{align}\label{I 2}
		\no &
		-a \sum_{\substack{m_{1}+m_{2}=k,\\1 \leq | m_{1} |}}  \langle \partial^{k} Q, \partial^{m_{1}} u \cdot \nabla \partial^{m_{2}}  Q  \rangle \\
		\no \lesssim&
		\sum_{\substack{m_{1}+m_{2}=k,\\1 \leq | m_{1} |}}  \|  \partial^{m_{1}} u \|_{L^4} \| \nabla \partial^{m_{2}} u \|_{L^4} \| \partial^{k} u \|_{L^2}
		+ \| \nabla u \|_{L^\infty} \| \partial^k \dot{Q} \|^2_{L^2}\\
		\no &
		+ \sum_{\substack{m_{1}+m_{2}=k,\\2 \leq | m_{1} |}}  \| \partial^{m_{1}} u\|_{L^4} \| \nabla \partial^{m_{2}} \dot{Q}\|_{L^4} \|\partial^{k} \dot{Q} \|_{L^2}
		+ \| \nabla u \|_{L^\infty} \| \partial^k Q \|^2_{L^2}\\
		\no &
		+ \sum_{\substack{m_{1}+m_{2}=k,\\2 \leq | m_{1} |}}  \| \partial^{m_{1}} u\|_{L^4} \| \nabla \partial^{m_{2}} Q\|_{L^4} \|\partial^{k} Q \|_{L^2} \\
		\no \lesssim&
		\| \nabla u \|_{H^s} \| u \|^2_{\dot{H}^s} + \| \nabla u  \|_{H^s} (\| \dot{Q} \|^2_{H^s} + \| Q \|^2_{H^s})\\
		\lesssim&
		E^{\frac{1}{2}}(t)D^{\frac{1}{2}}(t)(\|u\|_{\dot{H}^s}+\|Q\|_{H^s}+\|\dot{Q}\|_{H^s})\,.
	\end{align}
	
	We next estimate $I_{3}$.
	We observe that there are some dissipations in $I_3$ by employing the similar arguments as the derivations of $ L^{2} $-estimate. We can divide $I_{3}$ into two parts $ I^{e}_{3} $ and $ I^{m}_{3} $:
	\begin{equation}\nonumber
	\begin{aligned}
	I_{3}=I^{e}_{3}+I^{m}_{3} \,,
	\end{aligned}
	\end{equation}
	where
	\begin{equation}\nonumber
	\begin{aligned}
	I^{e}_{3}=&
	-\beta_{1} \langle Q \tr(Q \partial^{k} A) ,  \nabla \partial^{k} u \rangle
	-\beta_{5} \langle \partial^{k} A Q ,  \nabla \partial^{k} u \rangle
	-\beta_{6} \langle Q \partial^{k} A ,  \nabla \partial^{k} u \rangle\\
	&
	+\tfrac{\mu_{2}}{2} \langle [\partial^{k} \Omega , Q ],  \nabla \partial^{k} u \rangle
	-\tfrac{\mu_{2}}{2} \langle \partial^{k} \dot{Q},  \nabla \partial^{k} u \rangle
	+ \tfrac{\tilde{\mu}_{2}}{2} \langle \partial^{k} A  , \partial^{k} \dot{Q} \rangle\\
	&
	+\mu_{1} \langle [Q,[\partial^{k} \Omega, Q ] ],  \nabla \partial^{k} u \rangle \,,
	\end{aligned}
	\end{equation}
	and
	\begin{align}\nonumber
		I^{m}_{3} = & \underbrace{ - \beta_{1} \sum_{ \substack{ m_{1} + m_{2} + m_{3} = k , \\ 0 \leq | m_{3} | \leq |k|-1  } } \langle \partial^{m_{1} } Q \tr ( \partial^{m_{2}} Q \partial^{m_{3}} A ) , \nabla \partial^{k} u \rangle }_{I_{31}^m}  \\
		\no & \underbrace{ - \beta_{5} \sum_{ \substack{ m_{1} + m_{2} = k , \\ 0 \leq | m_{1} | \leq |k|-1 } } \langle \partial^{ m_{1} } A \partial^{m_{2}} Q ,  \nabla \partial^{k} u \rangle }_{I_{32}^m} \ \underbrace{ -\beta_{6} \sum_{ \substack{ m_{1} + m_{2} = k , \\ 0 \leq | m_{2} | \leq |k|-1 } } \langle \partial^{m_{1}} Q \partial^{m_{2}} A, \nabla \partial^{k} u \rangle }_{I_{33}^m} \\
		\no & + \underbrace{ \tfrac{\mu_{2}}{2} \sum_{ \substack{ m_{1} + m_{2} = k, \\ 0 \leq | m_{1} | \leq |k|-1}} \langle [ \partial^{m_{1}} \Omega , \partial^{m_{2}} Q ], \nabla \partial^{k} u \rangle }_{I_{34}^m} \\
		\no & + \underbrace{ \mu_{1} \sum_{ \substack{ m_{1} + m_{2} + m_{3} = k , \\ 0 \leq | m_{2} | \leq  |k|-1}} \langle [ \partial^{m_{1}} Q , [ \partial^{m_{2}} \Omega , \partial^{m_{3}} Q ] ], \nabla \partial^{k} u \rangle }_{I_{35}^m} \,.
	\end{align}
	
	From the similar derivations of the $ L^{2} $-estimate, one easily deduces that
	\begin{equation}\no
	\begin{aligned}
	I^{e}_{3} - \mu_{1} & \| \partial^k \dot{Q} \|^2_{L^{2}} -\tfrac{\beta_{4}}{2} \| \nabla \partial^{k} u \|^2_{L^2}  \\
	= & - \beta_{1} \|  Q : \partial^{k} A  \|^{2}_{L^{2}} - \int_{\R^d} F ( \nabla \partial^k u, \partial^k \dot{Q} , [ \partial^k \Omega, Q ] ) - ( \beta_{5} + \beta_{6}) \langle \partial ^ k A Q , \partial ^ k A  \rangle \,,
	\end{aligned}
	\end{equation}
	where $F (\cdot, \cdot , \cdot) : \R^{d \times d} \times \R^{d \times d} \times \R^{d \times d} \rightarrow \R$ is given in \eqref{Condition-H}.
	
	Now we will estimate the $I^{m}_{3}$ one by one, which should be controlled by the free energy or/and dissipation energy.
	Thanks to the H\"older inequality and the Sobolev embedding theory, one has
	\begin{equation}\nonumber
	\begin{aligned}
	I^{m}_{31}
	\lesssim&
	\| \partial^{k} Q \|_{L^{6}} \| Q \|_{L^{6}} \| \nabla u \|_{L^{6}} \| \nabla \partial^{k}u \|_{L^{2}}
	\\
	&+ \sum_{\substack{ m_{1} + m_{2} + m_{3} = k,\\ 1 \leq  | m_{3} | \leq |k|-1}} \|\partial^{m_{1}} Q \|_{L^{\infty }} \| \partial^{m_{2}} Q \|_{L^{\infty }}\| \nabla \partial^{m_{3}} u\|_{L^{2}} \| \nabla \partial^{k}u \|_{L^{2}}\\
	\lesssim&
	(\|  Q \|_{H^{s}} + \| \nabla Q \|_{H^{s}} )\| \nabla Q \|_{H^{s}} \| u \|_{\dot{H}^{s}}\|\nabla u\|_{H^{s}} \,,
	\end{aligned}
	\end{equation}
	and
	\begin{equation}\nonumber
	\begin{aligned}
	I^{m}_{32}
	\lesssim&
	\big( \| \nabla u \|_{L^{4}} \| \partial^{k} Q \|_{L^{4}}
	+ \sum_{\substack{ m_{1} + m_{2} = k, \\ 1 \leq | m_{1} | \leq |k|-1 }} \| \nabla \partial^{m_{1}} u\|_{L^{2}} \|\partial^{m_{2}} Q \|_{L^{\infty}} \big) \| \nabla \partial^{k} u \|_{L^{2}}\\
	\lesssim &
	\| \nabla Q \|_{H^{s}} \|  u \|_{\dot{H}^{s}}  \| \nabla u \|_{H^{s}} \,,
	\end{aligned}
	\end{equation}
	and
	\begin{equation}\nonumber
	\begin{aligned}
	I^{m}_{33}
	\lesssim&
	\big( \| \nabla u \|_{L^{4}} \| \partial^{k} Q \|_{L^{4}}
	+ \sum_{\substack{ m_{1} + m_{2} = k, \\ 1 \leq | m_{1} | \leq |k|-1}} \| \nabla \partial^{m_{1}} u\|_{L^{2}} \|\partial^{m_{2}} Q \|_{L^{\infty}}
	\big) \| \nabla \partial^{k} u \|_{L^{2}}\\
	\lesssim &
	\| \nabla Q \|_{H^{s}} \|  u \|_{\dot{H}^{s}}  \| \nabla u \|_{H^{s}} \,,
	\end{aligned}
	\end{equation}
	and
	\begin{equation}\nonumber
	\begin{aligned}
	I^{m}_{34}
	\lesssim&
	\big( \| \nabla u \|_{L^{4}} \| \partial^{k} Q \|_{L^{4}}
	+ \sum_{\substack{ m_{1} + m_{2} = k, \\ 1 \leq | m_{1} | \leq |k|-1}} \| \nabla \partial^{m_{1}} u\|_{L^{2}} \|\partial^{m_{2}} Q \|_{L^{\infty}} \big) \| \nabla \partial^{k} u \|_{L^{2}}\\
	&\\
	\lesssim &
	\| \nabla Q \|_{H^{s}} \|  u \|_{\dot{H}^{s}}  \| \nabla u \|_{H^{s}} \,,
	\end{aligned}
	\end{equation}
	and
	\begin{equation}\nonumber
	\begin{aligned}
	I^{m}_{35}
	\lesssim&
	\| \partial^{k} Q \|_{L^{6}} \| Q \|_{L^{6}} \| \nabla u \|_{L^{6}}  \| \nabla \partial^{k}u \|_{L^{2}}
	\\
	&+ \sum_{\substack{ m_{1} + m_{2} + m_{3} = k,\\ 1 \leq  | m_{3} | \leq |k|-1}} \|\partial^{m_{1}} Q \|_{L^{\infty }} \| \partial^{m_{2}} Q \|_{L^{\infty }}\| \nabla \partial^{m_{3}} u\|_{L^{2}}  \| \nabla \partial^{k}u \|_{L^{2}}\\
	\lesssim&
	(\|  Q \|_{H^{s}} + \| \nabla Q \|_{H^{s}} )\| \nabla Q \|_{H^{s}} \| u \|_{\dot{H}^{s}}\|\nabla u\|_{H^{s}} \,.
	\end{aligned}
	\end{equation}
	As a result, we have the following estimate, there is a constant $C>0$ such that
	\begin{equation}\label{I 3}
	\begin{aligned}
	I_{3} \leq & \mu_{1} \| \partial^k \dot{Q} \|^2_{L^{2}}
	- \beta_{1} \|  Q : \partial^{k} A  \|^{2}_{L^{2}} - \int_{\R^d} F ( \nabla \partial^k u, \partial^k \dot{Q} , [ \partial^k \Omega, Q ] ) \\
	& - ( \beta_{5} + \beta_{6}) \langle \partial ^ k A Q , \partial ^ k A  \rangle  + C I^{\prime}
	\,.
	\end{aligned}
	\end{equation}
	where
	\begin{equation}\nonumber
	\begin{aligned}
	I^{\prime} = (1+E^{\frac{1}{2}}(t))E^{\frac{1}{2}}(t)D^{\frac{1}{2}}(t) (\| u \|_{\dot{H}^{s}}  + \| Q \|_{H^{s}} +\|\nabla Q\|_{H^{s}})
	\,.
	\end{aligned}
	\end{equation}
	We next deal with $I_{4}$, it is easy to derive from Lemma \ref{Lemma3} and the Sobolev embedding theory that
	\begin{equation}\label{I 4}
	\begin{aligned}
	I_{4}
	=& -\mu_{1} \langle  \partial ^ k ([Q , \dot{Q}]), \nabla \partial ^ k u \rangle +\mu_{1} \langle \partial^{k} ([ \Omega , Q ] ), \partial^{k} \dot{Q} \rangle \\
	\lesssim&
	\| \partial^k ([Q , \dot{Q}]) \|_{L^2} \| \nabla \partial ^ k u  \|_{L^2} +  \| \partial^k  ([ \Omega , Q ] ) \|_{L^2} \|  \partial^{k} \dot{Q} \|_{L^2} \\
	\lesssim &
	\| \nabla u \|_{H^s} \|Q \|_{\dot{H}^s} \| \dot{Q} \|_{\dot{H}^s} \lesssim E^{\frac{1}{2}}(t)D^{\frac{1}{2}}(t)\|Q \|_{\dot{H}^s}\,.
	\end{aligned}
	\end{equation}
	As for the estimate of $I_5$, we only need to estimate the term
	\begin{equation}\no
	\begin{aligned}
	\langle \partial^{k} Q^{2},  \partial^{k} \dot{Q} \rangle \lesssim \sum_{ \substack{ m_{1} + m_{2} = k } }\| \partial^{m_{1}} Q \|_{L^4} \| \partial^{m_{2}} Q \|_{L^4} \| \partial^k \dot{Q} \|_{L^2} \lesssim \| \nabla Q \|^2_{H^s} \| \dot{Q} \|_{H^s}
	\,,
	\end{aligned}
	\end{equation}
	Then we have
	\begin{equation}\label{I 5}
	\begin{aligned}
	I_{5} \lesssim
	\| \nabla Q \|^2_{H^s} \| \dot{Q} \|_{H^s} \lesssim E^{\frac{1}{2}}(t)D^{\frac{1}{2}}(t)\|\nabla Q \|_{H^s}\,.
	\end{aligned}
	\end{equation}
	For the term $I_6$, we take advantage of the H\"older inequality and the Sobolev embedding inequality to get that
	\begin{equation}\label{I 6}
	\begin{aligned}
	I_6
	\lesssim & \sum_{ \substack{ m_{1} + m_{2} + m_{3} = k}}
	\| \partial^{m_{1}} Q \|_{L^6}  \| \partial^{m_{2}} Q \|_{L^6}  \| \partial^{m_{3}} Q \|_{L^6} \|  \partial^{k} \dot{Q} \|_{L^2} \\
	\lesssim &
	\|\nabla Q\|^3_{H^s} \|\dot{Q}\|_{H^s} \lesssim  E(t)D^{\frac{1}{2}}(t)\|\nabla Q\|_{H^s}\,.
	\end{aligned}
	\end{equation}
	From collecting the relations  \eqref{priori hk}, \eqref{I 1}, \eqref{I 2}, \eqref{I 3}, \eqref{I 4}, \eqref{I 5} and \eqref{I 6}, summing up for all $ 1 \leq | k | \leq  s $ and combining the $ L^{2} $-estimate \eqref{L2 sums}, we deduce that
	\begin{align}\label{priori 1}
		\no & \tfrac{1}{2} \tfrac{d}{dt} (  \| u \|^{2}_{H^{s}} + J\| \dot{Q} \|^{2}_{H^{s}} + L \| \nabla Q \|^{2}_{H^{s}} + a \| Q \|^{2}_{H^{s}}) + \beta_{1} \sum^{s}_{| k |=0} \|  Q : \partial^{k} A \|^{2}_{L^{2}}\\
		\no & + \sum_{|k|=0}^s \int_{\R^d} F ( \nabla \partial^k u, \partial^k \dot{Q} , [\partial^k \Omega, Q] ) d x  + ( \beta_{5} + \beta_{6} ) \sum^{s}_{| k |=0} \langle \partial ^ k A Q , \partial ^ k A  \rangle \\
		\lesssim &
		(1+E^{\frac{1}{2}}(t)) E^{\frac{1}{2}}(t)D^{\frac{1}{2}}(t) (\|u \|_{\dot{H}^s}+\|\dot{Q}\|_{H^{s}} + \|\nabla Q\|_{H^s} + \| Q \|_{H^s})\,,
	\end{align}
	
	From the definition $D (t)$ \eqref{Local-Dissipation} and $\mathcal{A}(t)$ \eqref{A}, the inequality \eqref{priori 1} reduces to
	\begin{equation}\label{Sum-1}
	\begin{aligned}
	& \tfrac{1}{2} \tfrac{d}{dt} E(t) + \beta_{1} \sum^{s}_{| k |=0} \|  Q : \partial^{k} A \|^{2}_{L^{2}} + \sum_{|k|=0}^s \int_{\R^d} F ( \nabla \partial^k u, \partial^k \dot{Q} , [\partial^k \Omega, Q] ) d x \\
	& + ( \beta_{5} + \beta_{6} ) \sum^{s}_{| k |=0} \langle \partial ^ k A Q , \partial ^ k A  \rangle
	\lesssim   (1+E^{\frac{1}{2}}(t)) E^{\frac{1}{2}}(t)D^{\frac{1}{2}}(t)\mathcal{A}^{\frac{1}{2}}(t)  \,.
	\end{aligned}
	\end{equation}
	
	If the Condition (H) in \eqref{Condition-H} holds, namely, there are two generic constants $\delta_0, \delta_1 \in (0,1]$ such that
	\begin{equation*}
		\begin{aligned}
			\sum_{|k|=0}^s \int_{\R^d} F ( \nabla \partial^k u, \partial^k \dot{Q} , [\partial^k \Omega, Q] ) d x \geq \delta_0 \tfrac{1}{2} \beta_4 \| \nabla u \|^2_{H^s} + \delta_1 \mu_1 \| \dot{Q} \|^2_{H^s} \,,
		\end{aligned}
	\end{equation*}
	we know that there is a constant $c_0^\star > 0$ such that
	\begin{equation}\label{F-1}
	\begin{aligned}
	\beta_{1} \sum^{s}_{| k |=0} \|  Q : \partial^{k} A \|^{2}_{L^{2}} + \sum_{|k|=0}^s \int_{\R^d} F ( \nabla \partial^k u, & \partial^k \dot{Q} , [\partial^k \Omega, Q] ) d x \geq c_0^\star D (t) \,.
	\end{aligned}
	\end{equation}
	
	If the Condition (H) does NOT hold, there are two cross terms
	$$- \mu_2 \sum_{|k| = 0}^s \langle \nabla \partial^k u , [\partial^k \Omega , Q] \rangle - \tfrac{1}{2} (\tilde{\mu}_2 - \mu_2) \sum_{|k| = 0}^s \langle \nabla \partial^k u , \partial^k \dot{Q} \rangle $$
	cannot be absorbed by the positive terms (with coefficients $\tfrac{1}{2} \beta_4$ and $\mu_1$ in $F (\nabla \partial^k u, \partial^k \dot{Q} , [\partial^k \Omega, Q])$). However, they can be bounded by
	\begin{equation}\label{F-2}
	\begin{aligned}
	\mu_2 \sum_{|k| = 0}^s \langle \nabla \partial^k u , [\partial^k \Omega , Q] \rangle + \tfrac{1}{2} (\tilde{\mu}_2 - \mu_2) \sum_{|k| = 0}^s \langle \nabla \partial^k u , \partial^k \dot{Q} \rangle \\
	\lesssim \| Q \|_{H^s} \| \nabla u \|^2_{H^s} + |\tilde{\mu}_2 - \mu_2| \| \dot{Q} \|_{H^s} \| \nabla u \|_{H^s} \\
	\lesssim E^\frac{1}{2} (t) D (t) + \eps \| \nabla u \|^2_{H^s} + |\tilde{\mu}_2 - \mu_2|^2 E (t) \,,
	\end{aligned}
	\end{equation}
	where $\eps > 0$ is small to be determined and the first inequality is derived from the Sobolev embedding theory. Moreover, there is a constant $c_1^\star > 0$ such that
	\begin{equation}\label{F-3}
	\begin{aligned}
	\sum_{|k|=0}^s \int_{\R^d} (\tfrac{\beta_{4}}{2}|\nabla\partial^k u|^2+\mu_{1}|\partial^k \dot{Q}|^2+|[\partial^k \Omega, Q]|^2) d x + \mu_2 \sum_{|k| = 0}^s \langle \nabla \partial^k u , [\partial^k \Omega , Q] \rangle\\
	+ \tfrac{1}{2} (\tilde{\mu}_2 - \mu_2) \sum_{|k| = 0}^s \langle \nabla \partial^k u , \partial^k \dot{Q} \rangle + \beta_{1} \sum^{s}_{| k |=0} \|  Q : \partial^{k} A \|^{2}_{L^{2}}  \\
	\geq c_1^\star D (t) - C |\tilde{\mu}_2 - \mu_2|^2 E (t)  \,.
	\end{aligned}
	\end{equation}
	Here we take $ 0 < \eps < \tfrac{1}{4} \beta_4 $ in \eqref{F-2}. Furthermore, one has
	\begin{equation}\label{F-4}
	\begin{aligned}
	- ( \beta_{5} + \beta_{6} ) \sum^{s}_{| k |=0} \langle \partial ^ k A Q , \partial ^ k A  \rangle \lesssim &|\beta_5 + \beta_6| \| Q \|_{H^s} \| \nabla u \|^2_{H^s} \\
	\lesssim &|\beta_5 + \beta_6| E^\frac{1}{2} (t) D (t) \,.
	\end{aligned}
	\end{equation}
	We therefore deduce from the inequalities \eqref{Sum-1}, \eqref{F-1}, \eqref{F-2}, \eqref{F-3} and \eqref{F-4} that
	\begin{equation}\label{H-L}
	\begin{aligned}
	\tfrac{1}{2} \tfrac{d}{d t} E (t) + c_0 D (t) \lesssim & ( \delta_H + |\beta_5 + \beta_6| ) E^\frac{1}{2} (t) D (t)     \\
	& + (1+E^{\frac{1}{2}}(t)) E^{\frac{1}{2}}(t)D^{\frac{1}{2}}(t)\mathcal{A}^{\frac{1}{2}}(t) + \delta_H |\tilde{\mu}_2 - \mu_2|^2 E (t) \,,
	\end{aligned}
	\end{equation}
	where $c_0 = \min\{ c_0^\star, c_1^\star \} > 0$, and
	\begin{equation*}
		\delta_H =
		\left\{
		\begin{array}{l}
			0 \,, \ \textrm{ if the Condition (H) holds} \,, \\
			1 \,, \ \textrm{ if the Condition (H) does not hold} \,.
		\end{array}
		\right.
	\end{equation*}
	Consequently, the proof of Lemma \ref{priori} is completed.
\end{proof}

\section{ Local well-posedness: proofs of Theorem \ref{Main-Thm-1} and Theorem \ref{Main-Thm-1-prime}}\label{Sec: Local}

In this section, we prove the local results in Theorem \ref{Main-Thm-1} and Theorem \ref{Main-Thm-1-prime} by employing the mollifier method.
\begin{proof}[Proof of Theorem \ref{Main-Thm-1} and Theorem \ref{Main-Thm-1-prime}]
	We first define the mollifying operator
	\begin{equation}
	\begin{aligned}
	\J_{\epsilon} f: = {\mathcal{F}}^{-1} ({\mathbf{1}}_{\mid \xi \mid  \leq  \frac{1}{\epsilon}} {\mathcal{F}} (f)),
	\end{aligned}
	\end{equation}
	where the symbol ${\mathcal{F}}$ is the Fourier transform operator and
	${\mathcal{F}}^{-1}$ is its inverse transform.
	It is easy to verify that the mollifier operator has the
	property $ { \J }^{2}_{\epsilon} = { \J }_{\epsilon}$.
	Then the approximate system is constructed as follows:
	\begin{equation}\label{App-Eq}
	\begin{cases}
	\partial_{t} u^{\epsilon}
	+ {\mathcal{P}} {\J}_{\epsilon} ({\J}_{\epsilon}
	u^{\epsilon} \cdot {\J}_{\epsilon} \nabla u^{\epsilon})
	- \frac{1}{2} \beta_{4} \Delta {\J}_{\epsilon}  u^{\epsilon}
	= {\mathcal{P}}\div( \Sigma_{1}^{\epsilon} + \Sigma_{2}^{\epsilon} +\Sigma_{3}^{\epsilon}) \\
	\ \ \ \ \ \ \ \ \ \ \ \ \ \ \ \ \ \ \ \ \ \ \ \ \ \ \ \ \div u^{\epsilon}=0,\\
	J {\J}_{\epsilon}\ddot{Q}^\epsilon + \mu_{1} {\J}_{\epsilon}\dot{Q}^{\epsilon}
	=
	L \Delta {\J}_{\epsilon}  Q^{\epsilon}
	-a Q^{\epsilon}
	+b {\J}_{\epsilon} ({\J}_{\epsilon} Q^{\epsilon} {\J }_{\epsilon} Q^{\epsilon})
	-b {\J }_{\epsilon} \tr( {\J }_{\epsilon} Q^{\epsilon} {\J }_{\epsilon} Q^{\epsilon}) \frac{I_{d}}{d} \\
	\ \ \ \ \ \ \ \ \ \ \ \ \ \ \ \ -c {\J}_{\epsilon} \big( {\J}_{\epsilon} Q^{\epsilon}
	\tr({\J}_{\epsilon} Q^{\epsilon}
	{\J}_{\epsilon} Q^{\epsilon}) \big)
	+\frac{ \tilde{\mu}_{2} }{2}  {\J}_{\epsilon} A^{\epsilon}
	+\mu_{1} {\J}_{\epsilon} [ {\J}_{\epsilon} \Omega^{\epsilon} ,
	{\J}_{\epsilon} Q^{\epsilon} ] \,.
	\end{cases}
	\end{equation}
	where
	\begin{equation}\nonumber
	\begin{aligned}
	\Sigma_{1}^{\epsilon}   & : = - L
	{\J}_{\epsilon} ( \nabla {\J}_{\epsilon}
	Q^{\epsilon} \odot \nabla {\J}_{\epsilon}  Q^{\epsilon}) \,, \\
	\Sigma_{2}^{\epsilon} & : =
	\beta_{1} {\J}_{\epsilon} \{{\J}_{\epsilon} Q^{\epsilon}
	\tr({\J}_{\epsilon} Q^{\epsilon} {\J}_{\epsilon} A^{\epsilon}) \}
	+\beta_{5} {\J}_{\epsilon} ( {\J}_{\epsilon} A^{\epsilon}
	{\J}_{\epsilon} Q^{\epsilon} )
	+\beta_{6} {\J}_{\epsilon} ({\J}_{\epsilon} Q^{\epsilon}
	{\J}_{\epsilon} A^{\epsilon}) \,,  \\
	\Sigma_{3}^{\epsilon} & : =
	\tfrac{ \mu_{2} }{2}  ({\J}_{\epsilon} \dot{Q}^{\epsilon}
	-{\J}_{\epsilon} [{\J}_{\epsilon} \Omega^{\epsilon} ,
	{\J}_{\epsilon} Q^{\epsilon} ] )
	+\mu_{1} {\J}_{\epsilon} [ {\J}_{\epsilon}Q^{\epsilon} ,
	({\J}_{\epsilon} \dot{Q}^{\epsilon}
	-[{\J}_{\epsilon} \Omega^{\epsilon} ,
	{\J}_{\epsilon} Q^{\epsilon}] ) ]  \,,
	\end{aligned}
	\end{equation}
	and
	\begin{equation}
	\begin{aligned}
	\Omega^{\epsilon}:&=\tfrac{1}{2} (\nabla u^{\epsilon} - \nabla^{\top} u^{\epsilon}),\\
	A^{\epsilon}:&=\tfrac{1}{2} (\nabla u^{\epsilon} + \nabla^{\top} u^{\epsilon}),\\
	\dot{Q}^{\epsilon}:&= \partial_{t}Q^{\epsilon} + {\mathcal{J}}_{\epsilon}
	({\mathcal{J}}_{\epsilon}u^{\epsilon} \cdot \nabla  {\mathcal{J}}_{\epsilon}Q^{\epsilon}),
	\end{aligned}
	\end{equation}
	and ${\mathcal{P}}$ denotes the Leray projector onto divergence-free vector fields. Moreover, the initial data of the approximate system \eqref{App-Eq} is imposed on
	\begin{equation}\label{IC-IAS}
	\begin{aligned}
	(  u^{\epsilon} , Q^{\epsilon} , \dot{Q}^{\epsilon} ) (0, x)  = ( {\mathcal{J}}_{\epsilon}u^{in} , {\mathcal{J}}_{\epsilon}Q^{in} , {\mathcal{J}}_{\epsilon}\tilde{Q}^{in} ) (x) \in  \R^d \times S^{(d)}_{0} \times S^{(d)}_{0} \,.
	\end{aligned}
	\end{equation}
	
	By ODE theory, we know that there is a maximal $T_{\epsilon}>0$ such that the approximate system
	\eqref{App-Eq} has a unique solution
	$u^{\epsilon}\in C([0,T_{\epsilon});H^{s}(\mathbb{R}^d))$ and
	$Q^{\epsilon}\in C([0,T_{\epsilon});H^{s+1}(\mathbb{R}^d))$.
	Notice  that the
	fact $ {\mathcal{J}}^{2}_{\epsilon} = {\mathcal{J}}_{\epsilon} $,
	we know $({\mathcal{J}}_{\epsilon}u^{\epsilon} , {\mathcal{J}}_{\epsilon} Q^{\epsilon})$
	is also a solution to the approximate system \eqref{App-Eq}.
	Then by the uniqueness of the solution
	we know that $( {\mathcal{J}}_{\epsilon} u^{\epsilon} , {\mathcal{J}}_{\epsilon} Q^{\epsilon})
	= (u^{\epsilon} , Q^{\epsilon})$.
	Therefore, the solution $(u^{\epsilon},Q^{\epsilon})$ also solves the following system
	\begin{equation}\label{App-Eq1}
	\begin{cases}
	\partial_{t} u^{\epsilon}
	+{\mathcal{P}} {\mathcal{J}}_{\epsilon} (u^{\epsilon} \cdot \nabla u^{\epsilon})
	-\frac{1}{2} \beta_{4} \Delta u^{\epsilon}
	=
	-L{\mathcal{P}}\div \Big( {\mathcal{J}}_{\epsilon} ( \nabla Q^{\epsilon} \odot \nabla Q^{\epsilon} ) \Big) \\
	\quad\quad\quad+ {\mathcal{P}}\div\Big( \beta_{1} {\mathcal{J}}_{\epsilon} \{ Q^{\epsilon} \tr (Q^{\epsilon}A^{\epsilon})\}
	+\beta_{5} {\mathcal{J}}_{\epsilon}( A^{\epsilon} Q^{\epsilon} )
	+\beta_{6} {\mathcal{J}}_{\epsilon}( Q^{\epsilon} A^{\epsilon}) \Big)\\
	\quad\quad\quad
	+{\mathcal{P}}\div\Big( \frac{ \mu_{2} }{2}  (\dot{Q}^{\epsilon} -{\mathcal{J}}_{\epsilon} [\Omega^{\epsilon} , Q^{\epsilon}] )
	+\mu_{1} {\mathcal{J}}_{\epsilon} [ Q^{\epsilon} , (\dot{Q}^{\epsilon} - [\Omega^{\epsilon} , Q^{\epsilon} ] )]\Big),\\
	\div u^{\epsilon} = 0,\\
	J \ddot{Q}^\epsilon+\mu_{1} \dot{Q}^{\epsilon} = L \Delta  Q^{\epsilon}
	-a Q^{\epsilon}
	+ b {\mathcal{J}}_{\epsilon} ( Q^{\epsilon} Q^{\epsilon} )
	-b \tr \Big( {\mathcal{J}}_{\epsilon} (Q^{\epsilon} Q^{\epsilon}) \Big) \frac{I_{d}}{d}\\
	\quad\quad\quad
	-c {\mathcal{J}}_{\epsilon} \Big( Q^{\epsilon} \tr( Q^{\epsilon} Q^{\epsilon} ) \Big)
	+\frac{ \tilde{\mu}_{2} }{2}  A^{\epsilon}
	+\mu_{1} {\mathcal{J}}_{\epsilon} [\Omega^{\epsilon} , Q^{\epsilon} ] \,.
	\end{cases}
	\end{equation}
	
	In the arguments proving the convergence ($\epsilon \rightarrow \infty$) of the approximate solutions \eqref{App-Eq1}-\eqref{IC-IAS}, it is essential to obtain uniform (in $\epsilon > 0$) energy estimates of \eqref{App-Eq1}-\eqref{IC-IAS}, whose derivations are the same as the derivations of the a priori estimates for the incompressible inertial Qian-Sheng model \eqref{CIQS} with initial data \eqref{initial date}. Another important thing is to prove existence of the uniform (in $\epsilon \geq 0$) lower bound of the lifespans $T_\epsilon > 0$ of the approximate system \eqref{App-Eq1}-\eqref{IC-IAS}. This is a standard process, which can be referred to \cite{DeAnna-Zarnescu-JDE-2018,Ma-2020DCDS} for instance. For simplicity, we will only establish the a priori estimate of \eqref{CIQS}-\eqref{initial date} in Section \ref{Sec: Apriori} and apply the continuity arguments.
	
	Based on Lemma \ref{priori}, we deduce from the Young's inequality and noticed that fact $\mathcal{A}(t)\lesssim \|u\|^2_{H^s} +\|\dot{Q}\|^2_{H^s}+ \| \nabla Q  \|^2_{H^s} + \|Q \|^2_{H^ s} \lesssim E(t)$ that
	\begin{equation}\label{Ener-Simple}
	\begin{aligned}
	\tfrac{d}{d t} E (t) + c_0 D (t)\lesssim &( \delta_H + |\beta_5 + \beta_6| ) E^\frac{1}{2} (t) D (t)\\
	&+ \delta_H |\tilde{\mu}_2 - \mu_2|^2 E (t)+ \big( 1 + E^{2} (t) \big) E (t) \,.
	\end{aligned}
	\end{equation}
	If $\beta_5 + \beta_6 = 0$ and the Condition (H) in \eqref{Condition-H} holds, i.e., $\delta_H = 0$, then the above inequality \eqref{Ener-Simple} tells us that
	\begin{equation}\label{Ener-Simple-1}
	\begin{aligned}
	\tfrac{d}{d t} E (t) + c_0 D (t) \lesssim  \big( 1 + E^{2} (t) \big) E (t) \,.
	\end{aligned}
	\end{equation}
	Let
	\begin{equation}\label{E(t)-bb}
	\begin{aligned}
	\mathbb{E} (t) : = E (t) + \int_0^t D (\tau) d \tau \,.
	\end{aligned}
	\end{equation}
	Since our goal is to construct the local-in-time solution, the assumption $t \leq 1$ makes sense. Thus, from \eqref{Ener-Simple-1}, we have
	\begin{equation}\label{Ener-Simple-2}
	\begin{aligned}
	\tfrac{d}{d t} \mathbb{E} (t) \lesssim \underbrace{  \Big( 1 + \mathbb{E}^{2} (t) \Big) \mathbb{E} (t) }_{\mathcal{G} ( \mathbb{E} (t) )}
	\end{aligned}
	\end{equation}
	for $0 \leq t \leq 1$. Noticing that $\mathcal{G} (w) \geq w \geq 0$ is strictly increasing on $w \geq 0$ and $\mathcal{G} (0) = 0$, we derive from the Gr\"onwall inequality that there is a constant $T = T (E^{in}, s, d, \, \textrm{all coefficients}) \in (0, 1]$ such that for all $t \in [0, T]$,
	\begin{equation*}
		\begin{aligned}
			\mathbb{E} (t)  \leq C_0 ( E^{in}, s, d, T, \, \textrm{all coefficients} ) \,,
		\end{aligned}
	\end{equation*}
	provided that $E^{in}<+\infty$, where $E^{in}$ is defined in \eqref{IC-Ener-1}. Thus, the existence result of Theorem \ref{Main-Thm-1} is proved.
	
	If $\beta_5 + \beta_6 \neq 0$ or the Condition (H) does not hold, the inequality \eqref{Ener-Loc} in Lemma \ref{priori} tells us that for some $C > 0$,
	\begin{equation}\label{Ener-2}
	\begin{aligned}
	\tfrac{d}{d t} E (t) + \tfrac{3}{2} c_0 D (t)
	\leq C E^\frac{1}{2} (t) D (t) + C \big( 1 + E^{2} (t) \big) E (t) \,,
	\end{aligned}
	\end{equation}
	where we utilize the Young's inequality. Recalling the definition $\mathbb{E} (t)$ in \eqref{E(t)-bb}, one immediately derives from \eqref{Ener-2} that
	\begin{equation}\label{Ener-3}
	\begin{aligned}
	\tfrac{d}{d t} \mathbb{E} (t) + \Big( \tfrac{1}{2} c_0 - C \mathbb{E}^\frac{1}{2} (t)  \Big) D (t) \leq C \big( 1 + \mathbb{E}^{2} (t) \big) \mathbb{E} (t) \,.
	\end{aligned}
	\end{equation}
	If we take $0 < \eps_1 \leq \tfrac{c_0^2}{8 C^2} > 0$ such that $\mathbb{E} (0) = E^{in} \leq \eps_1$, then
	\begin{equation*}
		\begin{aligned}
			C \mathbb{E}^\frac{1}{2} (0) \leq \tfrac{1}{2 \sqrt{2}} c_0 < \tfrac{1}{2} c_0  \,.
		\end{aligned}
	\end{equation*}
	We now introduce a number
	\begin{equation*}
		\begin{aligned}
			T^\star : = \sup \Big\{ \tau \in [0, 1] ; \sup_{0 \leq t \leq \tau} \mathbb{E} (t) \leq 2 \eps_1 = \tfrac{c_0^2}{4 C^2} \Big\} \in [0, 1] \,.
		\end{aligned}
	\end{equation*}
	Due to the continuity of $\mathbb{E} (t)$ on $t$, we easily conclude that $T^\star > 0$. Actually, the meaning of $T^\star$ is such that $ C \mathbb{E}^\frac{1}{2} (t) \leq \tfrac{1}{2} c_0$ holds for all $t \in [0, T^\star] \subseteq [0, 1]$. Consequently, for $ 0 \leq t \leq T^\star $, we deduce from the inequality \eqref{Ener-3} that
	\begin{equation}\label{Et}
	\begin{aligned}
	\tfrac{d}{d t} \mathbb{E} (t) \leq C \big( 1 + \mathbb{E}^{2} (t) \big) \mathbb{E} (t) \leq 2 C \eps_1 ( 1 + 4 \eps_1^2 )\,.
	\end{aligned}
	\end{equation}
	We now take
	\begin{equation}
	\begin{aligned}
	T = \min \big\{ 1, \tfrac{1}{4 C ( 1 + 4 \eps_1^2 )} \big\} \in (0, 1] \,.
	\end{aligned}
	\end{equation}
	We claim that $T^\star \geq T$. Otherwise, if $0 < T^\star < T$, it follows from integrating the inequality \eqref{Et} over $[0, t] \subseteq [0, T^*]$ that
	\begin{equation}
	\begin{aligned}
	\mathbb{E} (t) \leq & \mathbb{E} (0) + 2 C \eps_1 ( 1 + 4 \eps_1^2 ) t \leq \eps_1 + 2 C \eps_1 ( 1 + 4 \eps_1^2 ) T \\
	\leq & \eps_1 + 2 C \eps_1 ( 1 + 4 \eps_1^2 ) \cdot \tfrac{1}{4 C ( 1 + 4 \eps_1^2 )} = \eps_1 + \tfrac{1}{2} \eps_1 = \tfrac{3}{2} \eps_1 < 2 \eps_1
	\end{aligned}
	\end{equation}
	for all $t \in [0, T^\star]$. Namely,
	$$ \sup_{0 \leq t \leq T^\star} \mathbb{E} (t) \leq \tfrac{3}{2} \eps_1 < 2 \eps_1 \,. $$
	Then the continuity of $\mathbb{E} (t)$ implies that there is a $\varsigma > 0$ such that
	$$ \sup_{0 \leq t \leq T^\star + \varsigma} \mathbb{E} (t) \leq 2 \eps_1 \,, $$
	which contradicts to the definition of $T^\star$. Consequently, we know that $T^\star \geq T > 0$, and
	\begin{equation*}
		\begin{aligned}
			\mathbb{E} (t) \leq C_0 : = \eps_1 + 2 C \eps_1 ( 1 + 4 \eps_1^2 ) T
		\end{aligned}
	\end{equation*}
	holds for all $t \in [0, T]$. Then the existence proof of Theorem \ref{Main-Thm-1-prime} is thereby completed.
	
	Regarding to the uniqueness of the classical solution, considering two solutions with the same initial data, then writing down the equation satisfied by their difference, with the energy bound in hand, the method to prove the difference is actually zero is rather standard. We omit the details here. We infer the readers to see \cite{DeAnna-Zarnescu-JDE-2018}, where they proved the uniqueness of the solutions to the incompressible inertial Qian-Sheng model.
	
	At the end of this section, we show that $ Q \in S^{(d)}_{0}$, i.e., $ Q^{\top} =Q $ and $ \tr(Q)=0$. It is obvious that if $ Q $ is a solution of the system \eqref{CIQS}-\eqref{initial date}, then so is $Q^{\top}$. Here we utilize the initial conditions $Q^{in}, \tilde{Q}^{in} \in S_0^{(d)}$. Then by the uniqueness of the solution, we know that $ Q^{\top} =Q $. It remains to prove $\tr(Q)=0$. We take the trace on the both side of the third equation of \eqref{CIQS}, and obtain
	\begin{equation}\label{tr}
	\begin{aligned}
	J \partial_{t} \tr(\dot{Q}) + J u \cdot \nabla \tr(\dot{Q}) =  -\mu_{1} \tr(\dot{Q})
	+ L \Delta \tr(Q) - a \tr(Q) - c \tr(Q) | Q |^{2}  \,,
	\end{aligned}
	\end{equation}
	where we used the fact that $ Q^{\top} = Q $, $ \Omega^{\top} = -\Omega $, and $ \tr( \Omega Q) = \tr(Q \Omega)$.
	Then we multiply \eqref{tr} by
	$ \tr( \dot{Q} ) $ and integrate by parts over $\mathbb{R}^d$. We therefore obtain
	\begin{equation}\nonumber
	\begin{aligned}
	& \tfrac{1}{2} \tfrac{d}{dt} \big(J \| \tr(\dot{Q}) \|^{2}_{L^2} + L \| \nabla \tr(Q) \|^{2}_{L^2} + a \| \tr(Q) \|^{2}_{L^2} \big) + \mu_{1} \| \tr(\dot{Q}) \|^{2}_{L^2} \\
	\lesssim & L \| \nabla \tr(Q) \|^{2}_{L^{2}} + a \| \nabla \tr(Q) \|_{L^{2}} \| \tr(Q) \|_{L^{2}} + \| Q \|^{2}_{L^{6}} \| \tr(Q) \|_{L^{6}} \| \tr(\dot{Q}) \|_{L^{2}} \\
	\lesssim & L \| \nabla \tr(Q) \|^{2}_{L^{2}} + a \| \nabla \tr(Q) \|_{L^{2}} \| \tr(Q) \|_{L^{2}} + \| \nabla \tr(Q) \|_{L^{2}} \| \tr(\dot{Q}) \|_{L^{2}} \\
	\lesssim & \tfrac{\mu_{1}}{2} \| \tr(\dot{Q}) \|^{2}_{L^{2}} + C( L \| \nabla \tr(Q) \|^{2}_{L^{2}} + a \| \tr(Q) \|^{2}_{L^{2}}) \,,
	\end{aligned}
	\end{equation}
	where $ C > 0 $ is a constant. From applying the Gr\"onwall inequality and using the fact that $ \tr(\dot{Q}) (0) = \tr(\tilde{Q}^{in}) = 0 $, $
	\tr(Q)(0) = \tr(Q^{in}) = 0 $, we derive that
	\begin{equation}\nonumber
	\begin{aligned}
	& 0 \leq ( J\| \tr(\dot{Q}) \|^{2}_{L^{2}} + L \| \nabla \tr(Q) \|^{2}_{L^{2}} + a \| \tr(Q) \|^{2}_{L^{2}} ) (t) \\
	\lesssim & (J \| \tr(\dot{Q}) \|^{2}_{L^{2}} + L \| \nabla \tr(Q) \|^{2}_{L^{2}} + a \| \tr(Q) \|^{2}_{L^{2}})(0) \exp ( C t ) = 0
	\end{aligned}
	\end{equation}
	for all $0 \leq  t \leq T $. Consequently, $ \tr(Q) = 0$ holds for all $ 0 \leq t \leq T $. Consequently, the proofs of Theorem \ref{Main-Thm-1} and Theorem \ref{Main-Thm-1-prime} are finished.
\end{proof}

\section{Global existence of classical solution}\label{Sec: Global}

In this section, we will prove the global classical solution to the system \eqref{CIQS}-\eqref{initial date}.
We note that the $ H^{s} $-estiamte \eqref{H-L} does not have enough dissipation to prove the global solution with small initial data. To overcome this difficulty, we shall seek some dissipation of the $ Q $.

For some small $ \eta >0$ to be determined, we introduce the following { \em instant energy functional} $\mathcal{E}_\eta (t)$ :
\begin{equation}\label{Ge}
\begin{aligned}
\mathcal{E}_\eta (t) =
\|u\|^2_{H^s}+J( 1 - \eta ) \| \dot{Q} \|^{2}_{H^{s}}+L\|\nabla Q\|^2_{H^s} \\
+(a-J\eta) \|Q\|^2_{H^s} +J\eta\| \dot{Q} + Q \|^{2}_{H^{s}} \,.
\end{aligned}
\end{equation}
If $\eta > 0$ is small enough, such as
\begin{equation}\label{eta12-bnd}
\begin{aligned}
0 < \eta \leq \tfrac{1}{2} \min \{ 1, \tfrac{a}{J}, \tfrac{\Xi}{C} \} \,,
\end{aligned}
\end{equation}
then there are two constants $ C_1, C_2 > 0 $ such that
\begin{equation}\label{EE-glb}
\begin{aligned}
C_1 E (t) \leq \mathcal{E}_\eta (t) \leq C_2 E (t) \,,
\end{aligned}
\end{equation}
where the energy functional $ E(t) $ is defined in \eqref{Ener}.
We further introduce the energy dissipative rate $ \mathscr{D} (t) $:
\begin{equation}\label{D-global}
\begin{aligned}
\mathscr{D} (t) = \| \nabla u \|^2_{H^s} + \| \dot{Q} \|^2_{H^s} + \| \nabla Q \|^2_{H^s} + \| Q \|^2_{H^s} \,.
\end{aligned}
\end{equation}

We then give the following lemma to construct the global solution to \eqref{CIQS}, of which the main aim is to seek a further dissipative structures on the $Q$.

\begin{lemma}\label{Global lemma}
	Let $(u,Q)$ be a local solution constructed in Theorem \ref{Main-Thm-1} or Theorem \ref{Main-Thm-1-prime}. If the coefficients relations are given in Theorem \ref{Main-Thm-1} and Theorem \ref{Main-Thm-1-prime} EXCEPT FOR the case $(\tilde{\mu}_2 - \mu_2)^2 \geq 8 \beta_4 \mu_1$, then there are positive constants  $ \eta_{0}, \Lambda_1 , C > 0 $, depending only on the all coefficients, $s$ and $d$, such that for all $0 < \eta \leq \eta_0$,
	\begin{equation}\label{Ener-glb}
	\begin{aligned}
	\tfrac{1}{2} \tfrac{d}{d t} \mathcal{E}_\eta (t) + \Lambda_1 \mathscr{D} (t) \leq C \big( 1 + \mathcal{E}_\eta^\frac{1}{2} (t) \big) \mathcal{E}_\eta^\frac{1}{2} (t) \mathscr{D} (t) \,.
	\end{aligned}
	\end{equation}
\end{lemma}

\begin{remark}
	The constant $\eta_0$ should satisfy the bound \eqref{eta12-bnd}, namely, $ 0 < \eta_0 \leq \tfrac{1}{2}\min \{ 1, \tfrac{a}{J}, \tfrac{\Xi}{C} \}$, so that the instant energy functional $\mathcal{E}_\eta (t)$ is equivalent to $E(t)$, i.e., \eqref{EE-glb} holds. Furthermore, if $(\tilde{\mu}_2 - \mu_2)^2 \geq 8 \beta_4 \mu_1 > 0$, regardless whether $ \beta_5 + \beta_6 = 0 $ is true or not, the quantity $ - \tfrac{1}{2} (\tilde{\mu}_2 - \mu_2) \sum_{|k| = 0}^s \langle \nabla \partial^k u , \partial^k \dot{Q} \rangle $ in the term $ \sum_{|k|=0}^s \int_{\R^d} F ( \nabla \partial^k u , \partial^k \dot{Q} , $ $[ \partial^k \Omega , Q ] ) d x $ can not be absorbed by the energy $E(t)$ and energy dissipative rate $\mathscr{D} (t)$. So, we fail to globally extend the small data local solution in this case constructed in Theorem \ref{Main-Thm-1-prime}.
\end{remark}

\begin{proof}
	For all $ | k | \leq s$, acting $\partial^{k}$
	on the third equation of the system \eqref{CIQS} and taking $L^{2}$-inner product with $ \partial^{k} Q $ yield that
	\begin{equation}\label{Global Q}
	\begin{aligned}
	& \tfrac{1}{2} \tfrac{d}{dt} \big( J\| \partial^{k} \dot{Q} + \partial^{k} Q \|^{2}_{L^{2}}
	- J\| \partial^{k}\dot{Q} \|^{2}_{L^{2}}
	- J\| \partial^{k} Q \|^{2}_{L^{2}} \big)
	+ L \| \nabla \partial^{k} Q\|^{2}_{L^{2}}\\
	&
	+ a \| \partial^{k} Q \|^{2}_{L^{2}}
	- J \| \partial^{k} \dot{Q}\|^2_{L^{2}}\\
	=&
	\underbrace{- J \langle  \partial^{k}(u \cdot \nabla Q) , \partial^{k} \dot{Q} \rangle
		- J \langle  \partial^{k}(u \cdot \nabla \dot{Q}), \partial^{k} Q \rangle }_{\tilde{R}_{1}}
	\underbrace{- \mu_{1}\langle \partial^{k}\dot{Q}, \partial^{k} Q \rangle}_{\tilde{R}_{2}}\\
	&
	\underbrace{+ b \langle \partial^{k} (Q^{2}-\tfrac{1}{d} | Q |^{2}I_{d})), \partial^{k} Q \rangle}_{\tilde{R}_{3}}
	\underbrace{-c \langle \partial^{k} ( Q | Q |^{2})  , \partial^{k} Q \rangle}_{\tilde{R}_{4}}\\
	&
	\underbrace{+ \tfrac{\tilde{\mu}_{2}}{2}
		\langle \partial^{k} A , \partial^{k} Q  \rangle}_{\tilde{R}_{5}}
	\underbrace{+ \mu_{1} \langle \partial^{k}[ \Omega , Q ] ) , \partial^{k} Q \rangle}_{\tilde{R}_{6}} \,.
	\end{aligned}
	\end{equation}
	where we used
	\begin{equation*}
		\begin{aligned}
			\langle  \partial^{k} \ddot{Q} , \partial^{k} Q \rangle
			=
			\partial_{t} \langle  \partial^{k} \dot{Q} , \partial^{k} Q \rangle
			-\| \partial^{k} \dot{Q} \|^{2}_{L^{2}}
			+ \langle \partial^{k} (u \cdot \nabla Q) , \partial^{k} \dot{Q} \rangle
			+ \langle \partial^{k} ( u \cdot \nabla \dot{Q}), \partial^{k} Q \rangle \,.
		\end{aligned}
	\end{equation*}
	
	Now we estimate the terms $\tilde{R}_i$ $( 1 \leq i \leq 9 )$ in \eqref{Global Q}.
	
	The term $\tilde{R}_{1}$ can be easily controlled as
	\begin{equation}\label{Global Q1}
	\begin{aligned}
	\tilde{R}_{1}=&
	J\langle \nabla u \cdot\partial^{k} Q , \partial^{k} \dot{Q} \rangle
	-J\sum_{\substack{m_{1} + m_{2} = k, \\ | m_{1}| \geq 1 }}
	\langle \partial^{m_{1}} u \cdot \nabla \partial^{m_{2}} Q , \partial^{k} \dot{Q} \rangle\\
	&
	-J\langle\partial^k u \cdot \nabla\dot{Q} , \partial^{k} Q \rangle
	-J\sum_{\substack{m_{1} + m_{2} = k, \\ 1 \leq| m_{1}| \leq |k|-1 }}
	\langle \partial^{m_{1}} u \cdot \nabla \partial^{m_{2}}\dot{Q}, \partial^{k}Q  \rangle
	\\
	\lesssim&
	\| \nabla u \|_{L^{\infty}} \| \nabla \partial^{k} Q \|_{L^{2}} \| \partial^{k} \dot{Q} \|_{L^{2}}
	+\sum_{ \substack{ m_{1} + m_{2} = k,\\ | m_{1} | \geq 1 }}
	\| \partial^{m_{1}} u \|_{L^{4}} \| \nabla \partial^{m_{2}} Q \|_{L^{4}} \| \partial^{k} \dot{Q} \|_{L^{2}}\\
	&
	+\|\partial^k u\|_{L^2}\|\nabla\dot{Q}\|_{L^\infty}\| \partial^{k} Q\|_{L^2}
	+\sum_{\substack{m_{1} + m_{2} = k, \\ 1 \leq| m_{1}| \leq |k|-1 }}
	\|\partial^{m_{1}} u\|_{L^3} \|\nabla \partial^{m_{2}}\dot{Q}\|_{L^2}, \|\partial^{k}Q\|_{L^6}
	\\
	\lesssim&
	(\| u \|_{H^{s}} \| \nabla Q \|_{H^{s}} +
	\| u \|_{H^{s}} \|Q\|_{H^{s}} + \| \nabla u \|_{H^{s}} \| \nabla Q \|_{H^{s}}) \| \dot{Q} \|_{H^{s}} \lesssim E^\frac{1}{2}(t)\mathscr{D} (t) \,,
	\end{aligned}
	\end{equation}
	where the H\"older inequality and Sobolev embedding theory are utilized.
	From utilizing of the Lemma \ref{Lemma3}, Sobolev embedding theory and the fact $\div u=0$, one can deduce that
	\begin{equation}\label{Global Q2}
	\begin{aligned}
	\tilde{R}_{2} \leq C \| \dot{Q} \|_{H^{s}} \| Q \|_{H^{s}}
	\leq  \tfrac{1}{4 K_\#} a \| Q \|^2_{H^s} + C \| \dot{Q} \|^2_{H^s} \,,
	\end{aligned}
	\end{equation}
	where $K_\# > 0$ is the number of all multi-indexes $k \in \mathbb{N}^d $ such that $|k| \leq s$.
	
	From the similar arguments on the term $I_{5}$ and $I_{6}$ in Section \ref{ape}, one easily derives the bounds of the quantities $\tilde{R}_3$ and $\tilde{R}_4$ that
	\begin{equation}\label{Global Q3}
	\begin{aligned}
	\tilde{R}_{3}\lesssim \| Q \|^{3}_{H^{s}} \lesssim E^\frac{1}{2}(t)\mathscr{D} (t) \,,
	\end{aligned}
	\end{equation}
	and
	\begin{equation}\label{Global Q4}
	\begin{aligned}
	\tilde{R}_{4}\lesssim \|\nabla Q\|^3_{H^{s}}\| Q\|_{H^{s}}\lesssim E(t)\mathscr{D} (t) \,,
	\end{aligned}
	\end{equation}
	
	For the term $\tilde{R}_{5}$, it is yielded that by using the H\"older inequality and the Sobolev embedding theory
	\begin{equation}\label{Global Q5}
	\begin{aligned}
	\tilde{R}_{5}\leq
	C \|\nabla u\|_{H^{s}}\| Q\|_{H^{s}}
	\leq \tfrac{1}{4 K_\#} a \| Q \|^2_{H^s} + C \| \nabla u \|^2_{H^s}\,,
	\end{aligned}
	\end{equation}
	where $K_\# > 0$ is given in the estimate of $\tilde{R}_2$.
	
	As to the term $\tilde{R}_{6}$, the following bound can be implied by using the H\"older inequality, the Sobolev embedding theory and Moser-type inequality in Lemma \ref{Lemma3}:
	\begin{equation}\label{Global Q6}
	\begin{aligned}
	\tilde{R}_{6}\lesssim\| \nabla u \|_{H^{s}} \| Q \|^2_{H^{s}} \lesssim E^\frac{1}{2} (t) \mathscr{D} (t) \,.
	\end{aligned}
	\end{equation}
	
	From substituting the estimates \eqref{Global Q1}, \eqref{Global Q2}, \eqref{Global Q3}, \eqref{Global Q4}, \eqref{Global Q5} and \eqref{Global Q6} into \eqref{Global Q},
	and summing up for $ 0  \leq | k | \leq s$, we infer that
	\begin{equation}\label{Global Qs}
	\begin{aligned}
	\tfrac{1}{2} \tfrac{d}{dt} \big( J\| \dot{Q} &+ Q \|^{2}_{H^{s}} - J\| \dot{Q} \|^{2}_{H^{s}} - J\| Q \|^{2}_{H^{s}}\big)
	+ L \| \nabla Q \|^{2}_{H^{s} }\\
	&+ \tfrac{a}{2}\| Q \|^{2}_{H^{s}} - C \big( \| \dot{Q} \|^{2}_{H^s} +\| \nabla u \|^{2}_{H^{s}} \big) \lesssim (E^\frac{1}{2} (t) +E(t) )\mathscr{D} (t) \,.
	\end{aligned}
	\end{equation}
	
	At the end of the proof, we close the global energy estimates. From the difination of $\mathscr{D} (t)$ in \eqref{D-global}, the functional $\mathcal{A}(t)$ defined in
	\eqref{A} can be bounded by
	\begin{equation}
	\begin{aligned}
	\mathcal{A}(t) =& \|u \|^2_{\dot{H}^s} +\|\dot{Q}\|^2_{H^s}+ \| \nabla Q \|^2_{H^s} + \| Q\|^2_{H^ s}\\
	\lesssim&  \|\nabla u \|^2_{H^s} +\|\dot{Q}\|^2_{H^s}+ \| \nabla Q \|^2_{H^s} + \|Q\|^2_{H^ s} =\mathscr{D} (t)\,.
	\end{aligned}
	\end{equation}
	We choose a positive number
	\begin{equation}\label{eta0}
	\begin{aligned}
	\eta_0 = \tfrac{1}{2} \min \{ 1, \tfrac{a}{J}, \tfrac{\Xi}{C} \}\,.
	\end{aligned}
	\end{equation}
	Then, for any fixed $\eta \in (0, \eta_0]$, we multiply \eqref{Global Qs} by $\eta$, and then
	add them to the inequality \eqref{Sum-1}. We therefore have
	\begin{equation}\label{Global de}
	\begin{aligned}
	&  \tfrac{1}{2} \tfrac{d}{d t} \big( \|u\|^2_{H^s}+J( 1 - \eta ) \| \dot{Q} \|^{2}_{H^{s}}+L\|\nabla Q\|^2_{H^s} +(a-J\eta) \|Q\|^2_{H^s} +J\eta\| \dot{Q} + Q \|^{2}_{H^{s}} \big) \\
	&  + \beta_{1} \sum^{s}_{| k |=0} \|  Q : \partial^{k} A \|^{2}_{L^{2}} + \sum_{|k|=0}^s \int_{\R^d} F ( \nabla \partial^k u, \partial^k \dot{Q} , [\partial^k \Omega, Q] ) d x \\
	& + ( \beta_{5} + \beta_{6} ) \sum^{s}_{| k |=0} \langle \partial ^ k A Q , \partial ^ k A  \rangle - C \eta  \big( \| \nabla u \|^2_{H^s} + \| \dot{Q} \|^2_{H^s} \big)
	+ L\eta \| \nabla Q \|^2_{H^s} + \tfrac{a}{2}\eta\| Q \|^2_{H^s} \\
	& \lesssim \big( 1 + E^\frac{1}{2} (t) \big)E^\frac{1}{2} (t) \mathscr{D} (t) \,.
	\end{aligned}
	\end{equation}
	Moreover, from the Sobolev embedding theory, one easily derives that
	\begin{equation}\label{Glb-1}
	\begin{aligned}
	- ( \beta_{5} + \beta_{6}) \sum^{s}_{| k |=0} \langle \partial ^ k A Q , \partial ^ k A  \rangle \lesssim \| Q \|_{H^s} \| \nabla u \|^2_{H^s} \lesssim E^\frac{1}{2} (t) \mathscr{D} (t) \,.
	\end{aligned}
	\end{equation}
	
	If $\tilde{\mu}_2 \neq \mu_2$ and the Condition (H), defined in \eqref{Condition-H}, holds, then
	\begin{equation}\label{Glb-2}
	\begin{aligned}
	\sum_{|k|=0}^s \int_{\R^d} F ( \nabla \partial^k u , \partial^k \dot{Q} , [ \partial^k \Omega , Q ] ) d x \geq \delta_0 \tfrac{1}{2} \beta_1 \| \nabla u \|^2_{H^s} + \delta_1 \mu_1 \| \dot{Q} \|^2_{H^s} \\
	\geq \Xi_1 ( \| \nabla u \|^2_{H^s} + \| \dot{Q} \|^2_{H^s} ) \,,
	\end{aligned}
	\end{equation}
	where $\delta_0, \delta_1 \in (0, 1]$ are given in \eqref{Condition-H} and $\Xi_1 = \min \{ \tfrac{1}{2} \delta_0 \beta_4 , \delta_1 \mu_1 \} > 0$.
	
	If $\tilde{\mu}_2 = \mu_2$, regardless of whether the Condition (H) and the entropy inequalities \eqref{Entropy-Coeffs-1}-\eqref{H} are true or not, we have
	\begin{equation}\label{Glb-3}
	\begin{aligned}
	& \sum_{|k|=0}^s \int_{\R^d} F ( \nabla \partial^k u , \partial^k \dot{Q} , [ \partial^k \Omega , Q ] ) d x \\
	= & \tfrac{1}{2} \beta_4 \| \nabla u \|^2_{H^s} + \mu_1 \| \dot{Q} \|^2_{H^s} + \mu_1 \sum_{|k|=0}^s \| [\partial^k \Omega , Q] \|^2_{L^2} - \mu_2 \sum_{|k|=0}^s \langle \partial^k A , [\partial^k \Omega , Q] \rangle \\
	\geq & \Xi_2 ( \| \nabla u \|^2_{H^s} + \| \dot{Q} \|^2_{H^s} ) - C \| \nabla u \|^2_{H^s} \| Q \|_{H^s} \\
	\geq & \Xi_2 ( \| \nabla u \|^2_{H^s} + \| \dot{Q} \|^2_{H^s} ) - CE^\frac{1}{2} (t) \mathscr{D} (t)
	\end{aligned}
	\end{equation}
	for some constant $C > 0$, where $\Xi_2 = \min \{ \tfrac{1}{2} \beta_4, \mu_1 \} > 0$ and the first inequality is derived from the Sobolev embedding theory. Let $\Xi = \min \{ \Xi_1, \Xi_2 \} > 0$.
	
	If $\tilde{\mu}_2 \neq \mu_2$ and the entropy inequality \eqref{H} holds, then there are $\delta_0 , \delta_1 \in (0, 1)$ such that
	\begin{equation}\label{Glb-4}
	\begin{aligned}
	& \sum_{|k|=0}^s \int_{\R^d} F ( \nabla \partial^k u , \partial^k \dot{Q} , [ \partial^k \Omega , Q ] ) d x \\
	= & \tfrac{1}{2} \beta_4 \| \nabla u \|^2_{H^s} + \mu_1 \| \dot{Q} \|^2_{H^s} + \mu_1 \sum_{|k|=0}^s \| [\partial^k \Omega , Q] \|^2_{L^2} \\
	& - \tfrac{1}{2} (\tilde{\mu}_2 - \mu_2) \langle \partial^k A, \partial^k \dot{Q} \rangle - \mu_2 \sum_{|k|=0}^s \langle \partial^k A , [\partial^k \Omega , Q] \rangle \\
	\geq & \delta_0 \tfrac{1}{2} \beta_4 \| \nabla u \|^2_{H^s} + \delta_1 \mu_1 \| \dot{Q} \|^2_{H^s} - C \| \nabla u \|^2_{H^s} \big( \| Q \|_{H^s} + \| Q \|^2_{H^s} \big) \\
	\geq & \Xi_3 ( \| \nabla u \|^2_{H^s} + \| \dot{Q} \|^2_{H^s} ) - C \big( 1+E^\frac{1}{2} (t) \big) E^\frac{1}{2} (t) \mathscr{D} (t)
	\end{aligned}
	\end{equation}
	for some constant $C > 0$, where $\Xi_3 = \min \{ \delta_0 \frac{1}{2} \beta_4 , \delta_1 \mu_1 \} > 0$.
	
	If $\tilde{\mu}_2 \neq \mu_2$ and the entropy inequality \eqref{H} does not hold, regardless whether $ \beta_5 + \beta_6 = 0 $ is true or not, the quantity $ - \tfrac{1}{2} (\tilde{\mu}_2 - \mu_2) \sum_{|k| = 0}^s \langle \nabla \partial^k u , \partial^k \dot{Q} \rangle $ in the term $ \sum_{|k|=0}^s \int_{\R^d} F ( \nabla \partial^k u , \partial^k \dot{Q} , $ $[ \partial^k \Omega , Q ] ) d x $ can not be absorbed by the energy $E(t)$ and energy dissipative rate $\mathscr{D} (t)$. So, we fail to globally extend the small data local solution in this case constructed in Theorem \ref{Main-Thm-1-prime}.
	
	We substitute the bounds \eqref{Glb-1} and \eqref{Glb-2} (or \eqref{Glb-3}, or \eqref{Glb-4}) into \eqref{Global de} and then obtain
	\begin{equation}\label{Global-1}
	\begin{aligned}
	\tfrac{1}{2} \tfrac{d}{d t} \big( \|u\|^2_{H^s}+J( 1 - \eta ) \| \dot{Q} \|^{2}_{H^{s}}+L\|\nabla Q\|^2_{H^s} +(a-J\eta) \|Q\|^2_{H^s} +J\eta\| \dot{Q} + Q \|^{2}_{H^{s}} \big)\\
	+ (\Xi- C \eta) ( \| \nabla u \|^2_{H^s} + \| \dot{Q} \|^2_{H^s} ) + L\eta \| \nabla Q \|^2_{H^s} + \tfrac{a}{2}\eta\| Q \|^2_{H^s}\\
	\lesssim \big( 1 + E^\frac{1}{2} (t) \big) E^\frac{1}{2} (t) \mathscr{D} (t)  \,,
	\end{aligned}
	\end{equation}
	where the coefficient relation $\beta_1 \geq 0$ is also used.
	
	We remark that the choice of $\eta_0$ in \eqref{eta0} is such that the coefficients appeared in the inequality \eqref{Global de} are positive and the bound \eqref{eta12-bnd} holds, which guarantees the validity of \eqref{EE-glb}. Furthermore, one easily deduces that
	there are constants $\Lambda_1, \Lambda_2 > 0$ such that
	\begin{equation}
	\begin{aligned}
	\Lambda_1 \mathscr{D} (t) = & \Lambda_1 \big(\| \nabla u \|^2_{H^s} + \| Q \|^2_{H^s} + \| \nabla Q \|^2_{H^s} + \| \dot{Q} \|^2_{H^s} \big) \\
	\leq & (\Xi- C \eta) \big( \| \nabla u \|^2_{H^s} + \| \dot{Q} \|^2_{H^s} \big)   + L\eta \| \nabla Q \|^2_{H^s} + \tfrac{a}{2} \eta \| Q \|^2_{H^s} \\
	\leq & \Lambda_2 \big( \| \nabla u \|^2_{H^s} + \| Q \|^2_{H^s} + \| \nabla Q \|^2_{H^s} + \| \dot{Q} \|^2_{H^s} \big) = \Lambda_2 \mathscr{D} (t) \,.
	\end{aligned}
	\end{equation}
	Recalling the definition of $\mathcal{E}_\eta$ in \eqref{Ge} and the equivalent relation \eqref{EE-glb}, we can derive from the inequality \eqref{Global de} that
	\begin{equation*}
		\begin{aligned}
			\tfrac{1}{2} \tfrac{d}{d t} \mathcal{E}_\eta (t) + \Lambda_1 \mathscr{D} (t) \lesssim \big( 1 + \mathcal{E}^\frac{1}{2}_\eta (t) \big) \mathcal{E}_\eta^\frac{1}{2} (t) \mathscr{D} (t) \,.
		\end{aligned}
	\end{equation*}
	Consequently, the proof of Lemma \ref{Global lemma} is finished.
\end{proof}

\subsection*{Proof the Theorem \ref{Main-Thm-2}: Global well-posedness}
We now apply the continuity arguments to prove the existence of global-in-time solution. From the equivalent relation \eqref{EE-glb}, we easily observe that
\begin{equation}\label{EE-glb-IC}
\begin{aligned}
C_1 E(0)=C_1E^{in}\leq \mathcal{E}_\eta(0)\leq C_2 E(0) = C_2E^{in} \,,
\end{aligned}
\end{equation}
where the initial energy $E^{in}$ is given in \eqref{IC-Ener-glb}. We choose
\begin{equation}
\begin{aligned}
\eps_2 = \min \bigg\{ 1, \tfrac{\Lambda_1^2}{16 C_2 ( 1 + C_2^\frac{1}{2} )^2 C^2} \bigg\} \in (0, 1 ] \,.
\end{aligned}
\end{equation}
Combining with \eqref{EE-glb-IC}, one can easily verify that if $E^{in} \leq \eps_2$, then
\begin{equation}\label{IC-1}
\begin{aligned}
C \big( 1 + \mathcal{E}_\eta^\frac{1}{2} (0) \big) \mathcal{E}_\eta^\frac{1}{2} (0) \leq \tfrac{1}{4} \Lambda_1 < \tfrac{1}{2} \Lambda_1 \,.
\end{aligned}
\end{equation}
We introduce a number
\begin{equation}
\begin{aligned}
T^* = \sup \Big\{ \tau>0 ; \sup_{ t \in [0 , \tau] } C \big( 1 + \mathcal{E}_\eta^\frac{1}{2} (t) \big) \mathcal{E}_\eta^\frac{1}{2} (t) \leq \tfrac{1}{2} \Lambda_1 \Big\} \geq 0 \,,
\end{aligned}
\end{equation}
where the constants $\Lambda_1, C>0$ are given in Lemma \ref{Global lemma}. Then the continuity of $\mathcal{E}_\eta (t)$ on $t$ and the initial small condition \eqref{IC-1} yield that $T^* > 0$.

We claim that $T^* = + \infty$. Indeed, if $T^* < + \infty$, the global energy estimate \eqref{Ener-glb} in Lemma \ref{Global lemma} tells us that
\begin{equation}\label{Glb}
\begin{aligned}
\tfrac{d}{d t} {\mathcal{E}}_{\eta} (t) + \Lambda_1 \mathscr{D} (t) \leq 0
\end{aligned}
\end{equation}
holds for all $ t \in [ 0 , T^* ]$, which derives from integrating on $[0, t] \subseteq [ 0, T^* ]$ that
$$ {\mathcal{E}}_{\eta} (t) \leq  {\mathcal{E}}_{\eta}(0) \leq  C_2 E^{in} \leq C_2 \eps_2 $$
for all $ t \in [0,T^*] $. Then we have
\begin{equation}
\begin{aligned}
\sup_{ t \in [ 0 , T^* ] } C \big( 1 + \mathcal{E}_\eta^\frac{1}{2} (t) \big) \mathcal{E}_\eta^\frac{1}{2} (t) \leq \tfrac{1}{4} \Lambda_1 < \tfrac{1}{2} \Lambda_1 \,.
\end{aligned}
\end{equation}
By using the continuity of the energy $ \mathcal{E}_\eta (t) $ on $t$ implies that there exists a small positive $ \varsigma > 0 $ such that
\begin{equation}
\begin{aligned}
\sup_{ t \in [ 0 , T^* + \varsigma ] } C \big( 1 + \mathcal{E}_\eta^\frac{1}{2} (t) \big) \mathcal{E}_\eta^\frac{1}{2} (t) \leq \tfrac{1}{2} \Lambda_1 \,,
\end{aligned}
\end{equation}
which contradicts to the definition of $T^*$. As a consequence, $T^* = + \infty$. Then we derive from integrating the inequality \eqref{Glb} on $t \in \R^+$ that
\begin{equation}
\begin{aligned}
\sup_{t \geq 0 } \mathcal{E}_\eta (t) + \Lambda_1 \int^{\infty}_{0} \mathscr{D} (t) d t \lesssim E^{in} \,,
\end{aligned}
\end{equation}
and the proof of Theorem \ref{Main-Thm-2} is completed.\\

\noindent{\bf Proof of the Theorem \ref{Main-Thm-Decay}: Decay estimate on $\T^d$.} Let $(u, Q)$ be the solution to system \eqref{CIQS} on $(t, x) \in \R^+ \times \T^d$ constructed in Theorem \ref{Main-Thm-2}. We first take integration of the first equation of \eqref{CIQS} over $x \in \T^d$. We thereby have
\begin{equation*}
	\begin{aligned}
		\frac{d}{d t} \int_{\T^d} u (t, x) d x = 0 \,,
	\end{aligned}
\end{equation*}
which, together with the initial conditions \eqref{IC-Avar}, implies that
\begin{equation}
\begin{aligned}
\int_{\T^d} u (t, x) d x = 0 \,.
\end{aligned}
\end{equation}
Then, from the Poincar\'e inequality and the definition of $\mathscr{D} (t)$ in \eqref{D-global}, we deduces that
\begin{equation}\label{Dec-1}
\begin{aligned}
\| u \|_{L^2 (\T^d)} \lesssim \| \nabla u \|_{L^2 (\T^d)} \lesssim  \mathscr{D}^\frac{1}{2} (t) \,.
\end{aligned}
\end{equation}
Then, it is easy to be derived from \eqref{EE-glb}, \eqref{D-global} and \eqref{Dec-1} that
\begin{equation}\label{Dec-5}
\begin{aligned}
\mathcal{E}_\eta (t) \lesssim \| u \|^2_{L^2 (\T^d)} + \mathscr{D} (t) \lesssim \mathscr{D} (t) \,.
\end{aligned}
\end{equation}
Since \eqref{Glb} hold for all $t \geq 0$, we obtain that for some constant $C_3 > 0$,
\begin{equation*}
	\begin{aligned}
		\tfrac{d}{d t} \mathcal{E}_\eta (t) + C_3 \mathcal{E}_\eta (t) \leq 0 \ (\, \forall \, t \geq 0 \, ) \,,
	\end{aligned}
\end{equation*}
which implies that
\begin{equation}\label{Dec-6}
\begin{aligned}
\mathcal{E}_\eta (t) \leq \mathcal{E}_\eta (0) e^{- C_3 t} \ (\, \forall \, t \geq 0 \, ) \,.
\end{aligned}
\end{equation}
Then, the decay estimate \eqref{Decay-Bnd} in Theorem \ref{Main-Thm-Decay} is directly followed from \eqref{EE-glb}, \eqref{EE-glb-IC} and \eqref{Dec-6}. Consequently, the proof of Theorem \ref{Main-Thm-Decay} is finished.

\section{Conclusions}

In this paper, using energy method, we prove local and global well-posedness of inertial (hyperbolic) Qian-Sheng model. We introduce the coefficients constraints coming from entropy inequality and a further Condition (H). When both entropy inequality and Condition (H) are satisfied, local existence with large initial data are proved. When any one of these two conditions is broken, we can only prove local existence for small data. Furthermore, to extend the local solutions globally, we require either the entropy inequality, or $\tilde{\mu}_2 = \mu_2$.

Regarding to the future work on Qian-Sheng model, besides the open questions listed in \cite{DeAnna-Zarnescu-JDE-2018}, a new question raised by this paper is: are these conditions (entropy inequality and Condition (H)) sharp? i.e. without these conditions, can we prove ill-posedness? for example, finite time blow-up?

\section*{Acknowledgment}

The author N. J. appreciate Prof. Chun Liu for his comments on the entropy inequality and energy dissipation, which are crucial to this paper. Prof. Zarnescu also share with him on $Q$-tensor model of liquid crystals, which is very helpful on the preparation of this paper. This work is supported by the grants from the National Natural Foundation of China under contract No.11971360 and No. 11731008.

\end{document}